\documentclass[11pt]{article}
\usepackage{amsmath,amssymb,amscd,color}
\newcommand{\bq}{\begin{equation}}
\newcommand{\eq}{\end{equation}}

\newcommand{\R}{{ \mathbb{R}  }}

\newcommand{\bke}[1]{\left( #1 \right)}

\newcommand{\norm}[1]{\left\Vert #1 \right\Vert}
\newcommand{\abs}[1]{\left| #1 \right|}

\begin{document}
\bibliographystyle{plain}


\newtheorem{assumption}{Assumption}
\newtheorem{definition}{Definition}
\newtheorem{defn}{Definition}
\newtheorem{lemma}{Lemma}
\newtheorem{proposition}{Proposition}
\newtheorem{theorem}{Theorem}
\newtheorem{cor}{Corollary}
\newtheorem{remark}{Remark}
\numberwithin{equation}{section}

\def\Xint#1{\mathchoice
   {\XXint\displaystyle\textstyle{#1}}%
   {\XXint\textstyle\scriptstyle{#1}}%
   {\XXint\scriptstyle\scriptscriptstyle{#1}}%
   {\XXint\scriptscriptstyle\scriptscriptstyle{#1}}%
   \!\int}
\def\XXint#1#2#3{{\setbox0=\hbox{$#1{#2#3}{\int}$}
     \vcenter{\hbox{$#2#3$}}\kern-.5\wd0}}
\def\ddashint{\Xint=}
\def\dashint{\Xint-}
\def\aint{\Xint\diagup}

\newenvironment{proof}{{\bf Proof.}}{\hfill\fbox{}\par\vspace{.2cm}}
\newenvironment{pfthm1}{{\par\noindent\bf
            Proof of Theorem \ref{main-thm1} }}{\hfill\fbox{}\par\vspace{.2cm}}
\newenvironment{pflem2}{{\par\noindent\bf
            Proof of Lemma \ref{main-thm2} }}{\hfill\fbox{}\par\vspace{.2cm}}
\newenvironment{pfthm4}{{\par\noindent\bf
Proof of Theorem \ref{main-thm4}
}}{\hfill\fbox{}\par\vspace{.2cm}}
\newenvironment{pfthm3}{{\par\noindent\bf
Proof of Theorem \ref{main-thm5}}}{\hfill\fbox{}\par\vspace{.2cm}}
\newenvironment{pfthm5}{{\par\noindent\bf
Proof of Theorem \ref{TH-partial} }}{\hfill\fbox{}\par\vspace{.2cm}}
\newenvironment{pfcor5}{{\par\noindent\bf
Proof of Corollary \ref{Local-inequality-2}
}}{\hfill\fbox{}\par\vspace{.2cm}}
\newenvironment{pflemsregular}{{\par\noindent\bf
            Proof of Lemma \ref{sregular}. }}{\hfill\fbox{}\par\vspace{.2cm}}

\newenvironment{spp}{{\par\noindent
            \textbf{The sketch of the proof of Proposition \ref{prop1}.}\quad}}{}

\newenvironment{main-pfs}{{\par\noindent\bf
           Proofs of Theorem \ref{thm:3D-existence}
           and Theorem \ref{thm:3D-L^infinite bound}.
          }}
           {\hfill\fbox{}\par\vspace{.2cm}}

\newenvironment{main-pfs-domains}{{\par\noindent\bf
           Proofs of Theorem \ref{bounded-domain-10}
           and Theorem \ref{bounded-domain-20}.
          }}
           {\hfill\fbox{}\par\vspace{.2cm}}


\title{Existence of global solutions for a Keller-Segel-fluid equations with
 nonlinear diffusion
}

\author{Yun-Sung Chung and Kyungkeun Kang}

\date{}
\maketitle

\bigskip

\noindent{\bf Abstract:}\ We consider a coupled system consisting of
the Navier-Stokes equations and a porous medium type of Keller-Segel
system that model the motion of swimming bacteria living in fluid
and consuming oxygen. We establish the global-in-time existence of
weak solutions for the Cauchy problem of the system in dimension
three. In addition, if the Stokes system, instead Navier-Stokes
system, is considered for the fluid equation, we prove that bounded
weak solutions exist globally in time.


\section{Introduction}

We study a mathematical model describing the dynamics of oxygen,
swimming bacteria, and viscous incompressible fluids in $\R^3$. More
precisely, we consider the Cauchy problem for the coupled
Keller-Segel-Navier-Stokes system in $\R^3\times [0,T)$ with $0<T$
\begin{equation}\label{eq:Chemotaxis}
        \mbox{(KS-NS)}\quad\left\{
            \begin{array}{l}
                \partial_t n+u\cdot\nabla n=\Delta
                n^{1+\alpha}-\nabla\cdot\left(\chi(c)n\nabla
                c\right),
                \\
                \vspace{-3mm}\\
                \partial_t c+u\cdot\nabla c=\Delta c-\kappa(c)n,
                \\
                 \vspace{-3mm}\\
                \partial_t u+\tau(u\cdot\nabla)u+\nabla p=\Delta
                u-n\nabla\phi,
                \qquad {\rm{div}}\, u=0,
            \end{array}
         \right.
\end{equation}
where $n,c,u$ and $p$ are the cell density, oxygen concentration,
velocity field and pressure of the fluid. Here $\alpha>0$ is a
positive constant such that  porous medium type equation of $n$ is
under our consideration.


%

It is known that the above system models the motion of swimming
bacteria, so called {\it Bacillus subtilis}, which live in fluid and
consume oxygen. 
The functions $\chi:\mathbb R\to\mathbb R$ and $\kappa:\mathbb
R\to\mathbb R$ represent the chemotactic sensitivity and consumption
rate of oxygen. The constant $\tau$ in the third equation is $0$ or
$1$. When $\tau=1,$ $u$ becomes the velocity vector of fluid solving
Navier-Stokes equation. If the fluid motion is so slow, one may
assume $\tau=0$ so that $u$ is the velocity vector satisfying Stokes
system.

The above system \eqref{eq:Chemotaxis} was proposed by Tuval
\emph{et al.} in \cite{TCDWKG} for the case $\alpha=0$, which can be
extended to the case $\alpha>0$ when the diffusion of bacteria is
viewed like movement in a porus medium (see e.g. \cite{CKK},
\cite{FLM}, \cite{Liu-Lorz},
\cite{TW_1}, \cite{TW_2}). Throughout this paper, we call the above
system a Keller-Segel-Navier-Stokes equations (KSNS) if $\tau=1$ and
Keller-Segel-Stokes system (KSS) in case that $\tau=0$. Compared to
the Keller-Segel model (KS) of porus medium type:
\begin{equation*}\label{eq:Keller-Segle}
        \mbox{(KS)}\quad\left\{
            \begin{array}{l}
                \partial_t n=\Delta
                n^{1+\alpha}-\nabla\cdot\left(\chi n\nabla
                c\right),
                \\
                \vspace{-3mm}\\
                \tau\partial_t c=\Delta c-c+n,
            \end{array}
         \right.
\end{equation*}
where $\chi$ is a positive constant and $\tau=0$ or $1$, we
emphasize that the chemical substance (oxygen) in
\eqref{eq:Chemotaxis} is consumed, rather than produced by the
bacteria. The existence of global-in-time bounded weak solution for
Cauchy problem to (KS) with arbitrary large initial data is
guaranteed when $\alpha>1/3$ for three dimensional case. Otherwise,
the solution may blow up in finite time unless sufficiently small
initial condition is assumed (see the results \cite{SK} and
\cite{IY}).

The main concern of this paper is to specify the values of $\alpha$
so that the global-in-time weak solutions for (KSNS) and bounded
weak weak solutions for (KSS) are established, under the some
conditions of $\chi$ and $\kappa$. The notions of weak and bounded
weak solutions mentioned above are defined as follows:
\begin{definition}\label{Def:weak}
(Weak solutions)\,\,Let $\alpha>0$ and $0<T<\infty$. A triple
$(n,c,u)$ is said to be a \emph{weak solution} of the system
(\ref{eq:Chemotaxis}) if the followings are satisfied:
\\
$\mathrm{(i)}$ $n$ and $c$ are non-negative functions and $u$ is a
vector function defined in $\mathbb R^3\times (0,T)$ such that
\[
 n(1+|x|+|\log n|)\in L^\infty(0,T;L^1(\mathbb
R^3)),\quad\nabla n^{\frac{1+\alpha}{2}}\in L^2(0,T;L^2(\mathbb
R^3)),
\]
\[
 n\in L^\infty(0,T;L^p(\mathbb
R^3)),\quad\nabla n^{\frac{p+\alpha}{2}}\in L^2(0,T;L^2(\mathbb
R^3)),\quad 1\leq p\leq 1+\alpha
\]

\[
c\in L^\infty(0,T;H^1(\mathbb R^3))\cap L^2(0,T;H^2(\mathbb
R^3)),\quad c\in L^{\infty}(\mathbb R^3\times [0,T)),
\]
\[
u\in L^\infty(0,T;L^2(\mathbb R^3)),\quad \nabla u\in
L^2(0,T;L^2(\mathbb R^3)),
\] $\mathrm{(ii)}$ $(n,c,u)$ satisfies the
equation
  (\ref{eq:Chemotaxis}) in the sense of distributions, namely,
    \begin{eqnarray*}
           \int_0^T\int_{\mathbb R^3}\left(-n\varphi_t+\nabla
            n^{1+\alpha}\cdot\nabla\varphi-nu\cdot\nabla\varphi-n\chi(c)\nabla
            c\cdot\nabla\varphi\right) dxdt=\int_{\mathbb R^3}
            n_0\varphi(\cdot,0)dx,\\
            \int_0^T\int_{\mathbb R^3}\left(-c\varphi_t+\nabla
            c\cdot\nabla\varphi-cu\cdot\nabla\varphi+n\kappa(c)\varphi\right) dxdt=\int_{\mathbb R^3}
            c_0\varphi(\cdot,0)dx,\\
            \int_0^T\int_{\mathbb R^3}\left( -u\cdot\psi_t+\nabla
            u\cdot\nabla\psi+\left(\tau(u\cdot\nabla)u\right)\cdot\psi+n\nabla\phi\cdot\psi
            \right)dxdt=\int_{\mathbb R^3}u_0\cdot\psi(\cdot,0)dx\\
    \end{eqnarray*}
for all test functions $\varphi\in C^\infty_0\left(\mathbb R^3\times
[0,T)\right)$ and $\psi\in C^\infty_0\left(\mathbb R^3\times
[0,T),\mathbb R^3\right)$ with $\nabla\cdot\psi=0.$
\end{definition}

\begin{definition}\label{Def:bdd-weak}
(Bounded weak solutions)\,\,Let $\alpha>0$ and $0<T<\infty.$ A
triple $(n,c,u)$ is said to be a \emph{bounded weak solution} of the
system (\ref{eq:Chemotaxis}) if $(n,c,u)$ is a weak solution in
Definition \ref{Def:weak} and furthermore satisfies the following:
For any $p\in [1,\infty)$ and $q\in [2,\infty)$
\begin{itemize}
\item[$\mathrm{(i)}$] 
$n\in L^\infty((0,T)\times\mathbb R^3)$,\quad
$\displaystyle\nabla
n^{\frac{\alpha+p}{2}}\in L^2(0,T;L^2(\mathbb R^3))$.
\item[$\mathrm{(ii)}$] $c\in L^q(0,T;W^{2,q}(\mathbb R^3))$,
\quad$c_t\in L^q(0,T;L^q(\mathbb R^3))$.
\item[$\mathrm{(iii)}$] $u\in L^q(0,T;W^{2,q}(\mathbb R^3))$,
\quad $u_t\in L^q(0,T;L^q(\mathbb R^3))$.
\end{itemize}
\end{definition}
Before we state our main results, we recall some known results in
case that $\alpha>0$ (compare to e.g. \cite{ckl-cpde}, \cite{CKL},
\cite{DLM}, \cite{Lorz} and \cite{W14} for the case that $\alpha=0$
and references therein).
%
%
%
%
It was proved in \cite{FLM} that weak solutions of (KSS) for bounded
domains exist in case that $\alpha\in (1/2, 1]$ in two dimensions or
in case that $\alpha\in (\frac{-5+\sqrt {217}}{12},1]$ in three
dimensions. It was also shown that if domain is the whole space,
i.e. $\R^2$ or $\R^3$, then weak solutions exists globally when
$\alpha=1$. In \cite{Liu-Lorz}, the exponent $\alpha$ is reduced up
to $1/3$ for the case that spatial domain is $\R^3$, under the
following assumptions on $\chi$ and $\kappa$:
\begin{equation}\label{W-condition}
\chi(c), \kappa(c), \chi'(c), \kappa'(c)\ge
0,\,\,\kappa(0)=0,\,\,\frac{(\chi(c)\kappa(c))'}{\chi(c)}\ge
0,\,\,\bke{\frac{\kappa(c)}{\chi(c)}}^{''}<0.
\end{equation}
For the case of bounded domains in dimension two, bounded weak
solutions of (KSS) are constructed in \cite{TW_1} for any
$\alpha>0$.  In case that fluid equation is the Navier-Stokes
equations, it was proved recently in \cite[Theorem 1.8]{CKK} that
bounded weak solutions of (KS-NS) exist for any $\alpha>0$ in
$\R^2$.

On the other hand, in three dimensional case, \cite{TW_2} considered
a special case of $\chi(c)=1$ and $\kappa(c)=c$ and showed that
global weak solutions of (KSS) exist whenever $\alpha>1/7$ for
bounded domains. For the case of $\R^3$, \cite{CKK} proved that
global bounded weak solutions of (KSS) exist when $\alpha>1/4$ under
the hypothesis either $\chi'(\cdot)\geq \chi_0>0$ or
$\kappa'(\cdot)\geq \kappa_0>0$, where $\chi_0$ and $\kappa_0$ are
positive constants. It was shown, very recently, in \cite{W} that in
case that $\alpha>1/6$, bounded weak solutions are constructed in a
bounded convex domain $\Omega\subset\R^3$ with smooth boundary with
more relaxed conditions on $\chi$ and $\kappa$ in general setting
(see \cite[Theorem 1.1]{W} for the details).

The main purpose of this paper is to establish existence of global
weak and global bounded weak solutions for the Cauchy problem of the
system (KS-NS) in $\R^3$ under rather relaxed conditions of $\chi$
and $\kappa$ compared to \eqref{W-condition}, and smaller value of
$\alpha$ ever known.

Before stating our results, we address some conditions for $\kappa.$
To preserve the non-negativity of the density of bacteria $n(x,t)$
and the oxygen $c(x,t)$ for $0<t<T,$ we need the condition
$\kappa(0)=0$ (see the proof of Lemma \ref{lemma:weak}). The
condition $\kappa(\cdot)\geq 0$ is also essential since the bacteria
is consuming the oxygen. Namely, we assume the following hypothesis:
\begin{equation}\label{CK-Dec25-10}
  \quad\kappa(\cdot)\geq 0\quad\mbox{and}\quad
  \kappa(0)=0.
\end{equation}

We also present two different types of further assumptions on
chemotactic sensitivity $\chi$ and consumption rate $\kappa$
together with the range of $\alpha$. The first one is related to
weak solutions.
\begin{assumption}\label{assume1-chi-kappa}
We suppose that $\chi,\kappa$ satisfy $\chi' \in L^{\infty}_{\rm
loc}$ and $\kappa\in L^\infty_{\rm loc}$ with \eqref{CK-Dec25-10}.
We assume further that one of the
following holds:
\begin{itemize}
\item[(i)] $\alpha>1/6$.
\item[(ii)] $\alpha>0$ and $\chi'(\cdot)\geq \chi_0$  for some constant $\chi_0>0.$
\item[(iii)] $\alpha>0$ and $\kappa'(\cdot)\geq \kappa_0$ for some constant $\kappa_0>0$.
\end{itemize}
\end{assumption}

Another hypothesis is prepared for bounded weak solutions.
\begin{assumption}\label{assume2-chi-kappa}
We suppose that $\chi,\kappa$ satisfy $\chi' \in L^{\infty}_{\rm
loc}$ and $\kappa\in L^\infty_{\rm loc}$ with \eqref{CK-Dec25-10}.
We assume further that one of the following holds:
\begin{itemize}
\item[(i)] $\alpha>1/6$.
\item[(ii)] $\alpha>1/8$ and $\chi'(\cdot)\geq \chi_0$  for some constant $\chi_0>0.$
\item[(iii)] $\alpha>1/8$ and $\kappa'(\cdot)\geq \kappa_0$ for some constant $\kappa_0>0$.
\end{itemize}
\end{assumption}

We are now in a position to state the main results. The first main
result of this paper is the existence of global-in-time weak
solution of the system (KS-NS) with $\tau=1$, which means that the
fluid equations are the Navier-Stokes equations. More precisely, the
first result reads as follows:
\begin{theorem}\label{thm:3D-existence}
Let $\tau=1$ and the Assumption \ref{assume1-chi-kappa} hold.
Suppose that initial data $(n_0,c_0,u_0)$ satisfies
\begin{equation}\label{initial-weak}
n_0(1+|x|+|\log n_0|)\in L^1(\mathbb R^3),\,\, n_0\in
L^{1+\alpha}(\mathbb R^3),\,\, c_0\in L^{\infty}(\mathbb R^3)\cap
H^1(\mathbb R^3),\,\, u_0\in L^2(\mathbb R^3).
\end{equation}
Then, for each $T>0$, there exists a weak solution $(n,c,u)$ for the
system \eqref{eq:Chemotaxis} and it satisfies
\begin{equation*}
\sup_{0\le t\le T}\bke{\int_{\mathbb R^3}n(t)\bke{\abs{\log
n(t)}+2\langle
x\rangle}~dx+\norm{n(t)}_{L^{1+\alpha}}^{1+\alpha}+\norm{\nabla
c(t)}_{L^2}^2 +\norm{u(t)}_{L^2}^2}
\end{equation*}
\begin{equation}\label{JK-Dec10-300}
+\int_0^T\bke{\norm{\nabla
n^\frac{1+\alpha}{2}(t)}_{L^2}^2+\norm{\nabla
n^\frac{1+2\alpha}{2}(t)}_{L^2}^2 +\norm{\Delta
c(t)}_{L^2}^2+\norm{\nabla u(t)}_{L^2}^2}dt<C,
\end{equation}
where $\langle x\rangle=(1+|x|^2)^\frac{1}{2}$ and $C=C(T,
\norm{n_0}_{L^1\cap L^{1+\alpha}}, \norm{n_0\log n_0}_{L^1},
\norm{n_0\langle x\rangle}_{L^1}, \norm{\nabla c_0}_{L^2},
\norm{u_0}_{L^2})$.
\end{theorem}

If the fluid equations are restricted to be the Stokes equations,
\emph{i.e.} $\tau=0,$ and if the range of $\alpha$ in the hypothesis
in Theorem \ref{thm:3D-existence} is a bit more restrictive, we can
construct bounded weak solutions for the Cauchy problem. More
precisely, our second result reads as follows:

\begin{theorem}\label{thm:3D-L^infinite bound}
Let $\tau=0$ and the Assumption \ref{assume2-chi-kappa} hold.
Suppose that initial data $(n_0,c_0,u_0)$ satisfy
\eqref{initial-weak} and for any $q\in [2, \infty)$
\begin{equation}\label{initial-boundweak}
n_0\in L^{\infty}(\mathbb R^3), \quad c_0\in W^{1,q}(\mathbb R^3),
\quad u_0\in W^{1,q}(\mathbb R^3).
\end{equation}
Then, for each $T>0$, there exists a bounded weak solution $(n,c,u)$
for the system \eqref{eq:Chemotaxis} such that
it satisfies
\begin{equation*}
\norm{n}_{L^{\infty}((0,T)\times \R^3)}+\norm{\nabla
n^{\frac{p+\alpha}{2}}}_{L^2((0,T)\times \R^3)}
+\norm{c}_{L^q(0,T;W^{2,q}(\mathbb R^3))}
\end{equation*}
\begin{equation}\label{JK-Dec10-500}
+\norm{\partial_t c}_{L^q(0,T;L^q(\mathbb R^3))}
+\norm{u}_{L^q(0,T;W^{2,q}(\mathbb R^3))}+\norm{\partial_t
u}_{L^q(0,T;L^q(\mathbb R^3))}<C,
\end{equation}
where $C=C(T, \norm{n_0}_{L^{\infty}}, \norm{c_0}_{W^{1,q}},
\norm{u_0}_{W^{1,q}})$.
\end{theorem}

\begin{remark}
The results in Theorem \ref{thm:3D-L^infinite bound} can be
rephrased as follows: In case that $\alpha>1/6$, the bounded weak
solutions can be constructed with the assumptions that $\chi,\kappa$
satisfy $\chi' \in L^{\infty}_{\rm loc}$ and $\kappa\in
L^\infty_{\rm loc}$ with $\kappa(\cdot)\geq 0$ and $\kappa(0)=0$
(compare to \cite{W} for bounded domain case). If we assume further
that either $\chi'(\cdot)\geq \chi_0>0$ or $\kappa'(\cdot)\geq
\kappa_0>0$, then bounded weak solutions exist for $\alpha>1/8$,
weaker than $\alpha>1/6$.
\end{remark}


%

Our main concern is construction of weak and bounded weak solutions
for the Cauchy problems of the system \eqref{eq:Chemotaxis} in
$\R^3$ but it can be extended without so much difficulty to bounded
domains with Neumann boundary conditions for $n$ and $c$ and no-slip
boundary conditions for $u$. To be more precise, let $\Omega\subset
\R^3$ be a bounded domain with smooth boundary and we consider the
system \eqref{eq:Chemotaxis} in $\Omega\times[0,T)$ with boundary
conditions
\begin{equation}\label{CK-Dec13-10}
\frac{\partial n}{\partial \nu}=\frac{\partial c}{\partial
\nu}=0,\qquad u=0\qquad \mbox{ on }\,\,\partial\Omega.
\end{equation}
We then obtain following results for the case of bounded domains by
following almost same arguments as in Theorem \ref{thm:3D-existence}
and Theorem \ref{thm:3D-L^infinite bound}.

%
%

\begin{theorem}\label{bounded-domain-10}
Let $\tau=1$ and the Assumption \ref{assume1-chi-kappa} hold with
replacement of $\R^3$ by a bounded domain $\Omega$ with smooth
boundary. Suppose that initial data $(n_0,c_0,u_0)$ satisfies
\begin{equation}
n_0\in L^1(\Omega)\cap L^{1+\alpha}(\Omega),\,\, c_0\in
L^{\infty}(\Omega)\cap H^1(\Omega),\,\, u_0\in L^2(\Omega).
\end{equation}
Then, for each $T>0$, there exists a weak solution $(n,c,u)$ for the
system \eqref{eq:Chemotaxis} and \eqref{CK-Dec13-10}, and it
satisfies
\begin{equation*}
\sup_{0\le t\le T}\bke{\int_{\mathbb R^3}n(t)\abs{\log
n(t)}dx+\norm{n(t)}_{L^{1+\alpha}}^{1+\alpha}+\norm{\nabla
c(t)}_{L^2}^2 +\norm{u(t)}_{L^2}^2}
\end{equation*}
\begin{equation*}
+\int_0^T\bke{\norm{\nabla
n^\frac{1+\alpha}{2}(t)}_{L^2}^2+\norm{\nabla
n^\frac{1+2\alpha}{2}(t)}_{L^2}^2 +\norm{\Delta
c(t)}_{L^2}^2+\norm{\nabla u(t)}_{L^2}^2}dt<C,
\end{equation*}
where $C=C(T, \norm{n_0}_{L^1\cap L^{1+\alpha}}, \norm{n_0\log
n_0}_{L^1}, \norm{\nabla c_0}_{L^2}, \norm{u_0}_{L^2})$.
\end{theorem}

\begin{theorem}\label{bounded-domain-20}
Let $\tau=0$ and the Assumption \ref{assume2-chi-kappa} hold with
replacement of $\R^3$ by a bounded domain $\Omega$ with smooth
boundary. Suppose that initial data $(n_0,c_0,u_0)$ satisfy
\eqref{initial-weak} and for any $q\in [2, \infty)$
\begin{equation*}
n_0\in L^{\infty}(\Omega), \quad c_0\in W^{1,q}(\Omega), \quad
u_0\in W^{1,q}(\Omega).
\end{equation*}
Then, for each $T>0$, there exists a bounded weak solution $(n,c,u)$
for the system \eqref{eq:Chemotaxis}  and \eqref{CK-Dec13-10} such
that
it
satisfies
\begin{equation*}
\norm{n}_{L^{\infty}((0,T)\times \Omega)}+\norm{\nabla
n^{\frac{p+\alpha}{2}}}_{L^2((0,T)\times \Omega)}
+\norm{c}_{L^q(0,T;W^{2,q}(\Omega))}
\end{equation*}
\begin{equation*}
+\norm{\partial_t c}_{L^q(0,T;L^q(\Omega))}
+\norm{u}_{L^q(0,T;W^{2,q}(\Omega))}+\norm{\partial_t
u}_{L^q(0,T;L^q(\Omega))}<C,
\end{equation*}
where $C=C(T, \norm{n_0}_{L^{\infty}}, \norm{c_0}_{W^{1,q}},
\norm{u_0}_{W^{1,q}})$.
\end{theorem}


Since verifications of Theorem \ref{bounded-domain-10} and Theorem
\ref{bounded-domain-20} are similar to those of Theorem
\ref{thm:3D-existence} and Theorem \ref{thm:3D-L^infinite bound},
only differences are indicated in the proofs in section 4.

\begin{remark}
Theorem \ref{thm:3D-existence} and Theorem \ref{thm:3D-L^infinite
bound} are improvements over the results in \cite[Theorem 1.5 and
Theorem 1.7]{CKK}. We also remark that Theorem
\ref{bounded-domain-20} covers the result of \cite{TW_2}, since the
case $\chi=1$ and $\kappa(c)=c$ is the special case of (iii) in the
Assumption \ref{assume2-chi-kappa} and besides, $\alpha>1/8$ is
wider than $\alpha>1/7$ in \cite{TW_2}.
\end{remark}

This paper is organized as follows. In section 2 and section 3 a
priori estimates of weak and bounded weak solutions are established.
In section 4 we present the proofs of main results.

\section{Weak solutions}

Throughout this section, we study the solutions of the approximate
problem of \eqref{eq:Chemotaxis} given by
\begin{equation}\label{eq:Chemotaxis-approximate}
        \left\{
            \begin{array}{l}
                \partial_t n_\varrho+u_\varrho\cdot\nabla n_\varrho=\Delta
                (n_\varrho+\varrho)^{1+\alpha}-\nabla\cdot\left(\chi(c_\varrho)n_\varrho\nabla
                c_\varrho\right),
                \\
                \vspace{-3mm}\\
                \partial_t c_\varrho+u_\varrho\cdot\nabla c_\varrho=\Delta c_\varrho-\kappa(c_\varrho)n_\varrho,
                \\
                 \vspace{-3mm}\\
                \partial_t u_\varrho+\tau(u_\varrho\cdot\nabla)u_\varrho+\nabla p_\varrho=\Delta
                u_\varrho-n_\varrho\nabla\phi,
                \qquad {\rm{div}}\, u_\varrho=0,
            \end{array}
         \right.
\end{equation}
in $\mathbb R^3\times (0,T)$ with smooth initial data
$(n_{0\varrho}, c_{0\varrho}, u_{0\varrho})$ given by
    \begin{align*}
        n_{0\varrho}=\psi_\varrho\ast n_0,\quad c_{0\varrho}=\psi_\varrho\ast c_0
        \quad\mathrm{and}\quad u_{0\varrho}=\psi_\varrho\ast u_0
    \end{align*}
where $\phi_\varrho$ denotes the usual mollifier with
$\varrho\in(0,1).$

It is known that, due to the standard theory of existence and
regularity as done in \cite{FLM} and \cite{TW_2}, there exists a
classical solution of the equation \eqref{eq:Chemotaxis-approximate}
locally in time for each $\varrho\in(0,1)$.
The main objective of this section is to derive appropriate uniform
estimates, independent of $\varrho$, of the solutions. The estimates
are crucially used in Section $3$ to extend the above local solution
to any given time interval (0,T) and to construct the weak solutions
and the bounded weak solutions of the equation
\eqref{eq:Chemotaxis}.

We start with some notations. For $1\leq q\leq \infty$, we denote by
$W^{k,q}(\mathbb R^3)$ the usual Sobolev spaces, namely
$W^{k,q}(\mathbb R^3)=\{f\in L^q(\mathbb R^3): D^{\alpha}f\in
L^q(\mathbb R^3), 0\leq \abs{\alpha}\leq k\}$. The set of $q-$th
power Lebesgue integrable functions on $\mathbb R^3$ is denoted by
$L^q(\mathbb R^3)$. In what follows, for simplicity, $\|\cdot\|_p$
denotes $\|\cdot\|_{L^p(\mathbb R^3)}$ for $1\leq p\leq \infty$,
unless there is any confusion to be expected. We also denote by
$W^{-k,q'}(\mathbb R^3)$ dual space of $W^{k,q}_0(\mathbb R^3)$,
where $q$ and $q'$ are H\"older conjugates. The letters
$C=C(*,...,*)$ and $C'=C'(*,...,*)$ are used to represent generic
constants, depending on $*,...,*$, which may change from line to
line.

We first recall that maximal regularity estimate of the
inhomogeneous heat equation, which we use later. Let $0<T<\infty$
and $1<p<\infty$, and we consider
\[
v_t-\Delta v=f \quad \mbox{ in }\,\R^3\times (0,T)
\]
with initial datum $v(x,0)=v_0(x)$ with $v_0\in W^{1,p}(\R^3)$. Then
following $L^p$ estimate is well-known:
\begin{equation}\label{heat-maximal}
\norm{v_t}_{L^p((0,T)\times\R^3)}+\norm{v}_{L^p((0,T);W^{2,p}(\R^3))}\leq
C\bke{\norm{f}_{L^p((0,T)\times\R^3)}+\norm{v_0}_{W^{1,p}(\R^3)}}.
\end{equation}

In the following lemma, we give an estimate of solutions of
(\ref{eq:Chemotaxis-approximate}) under the Assumption
\ref{assume1-chi-kappa}. The estimate is used in Section 3 to
construct the weak solution of the equation \eqref{eq:Chemotaxis}.
For the sake of simplicity, throughout Section $2$, we denote
$n_\varrho,~c_\varrho$ and $u_\varrho$ by $n,c$ and $u.$ Also, we
define the functionals $E_{M}(t)$ and $D(t)$ as follows:
    \begin{equation}\label{eq:E(t)}
        E_{M}(t):=\int_{\mathbb R^3}n(t)\bke{\abs{\log n(t)}+2\langle x\rangle}~dx+\norm{n(t)}_{1+\alpha}^{1+\alpha}+\norm{\nabla c(t)}_2^2
        +\frac{M+2}{2}\norm{u(t)}_2^2
    \end{equation}
and
    \begin{equation}\label{eq:D(t)}
        D(t):=\norm{\nabla n^\frac{1+\alpha}{2}(t)}_2^2+\norm{\nabla
        n^\frac{1+2\alpha}{2}(t)}_2^2
        +\norm{\Delta c(t)}_2^2+\norm{\nabla u(t)}_2^2,
    \end{equation}
where $\langle x\rangle=(1+|x|^2)^\frac{1}{2}$ and  $M$ is a
positive constant, which will be specified later.
\begin{lemma}\label{lemma:weak}
Let $T>0$. Suppose that $(n,c,u)$ is a classical solution for the
system (\ref{eq:Chemotaxis-approximate})$_\varrho,$
$\varrho\in(0,1)$ with the smooth initial datum $(n_{0\varrho},
c_{0\varrho}, u_{0\varrho})$ satisfies the initial condition
(\ref{initial-weak}) independently of $\rho$. Assume further that
$\chi,\kappa$ and $\alpha$ satisfy Assumption
\ref{assume1-chi-kappa}. Then, there exists $C>0$ and $M>0$, which
are independent of $\rho$, such that for any $0<t\le T$, $E_{M}(t)$
and $D(t)$ defined in \eqref{eq:E(t)}-\eqref{eq:D(t)} satisfy
\begin{equation}\label{eq:E'+D<CE}
\sup_{0\leq \tau\leq t}E_{M}(\tau)+\int_0^t D(\tau) d\tau< C,
\end{equation}
where $C=C(T, \norm{n_0}_{L^1}, \norm{n_0\log n_0}_{L^1},
\norm{n_0\langle x\rangle}_{L^1}, \norm{n_0}_{L^{1+\alpha}},
\norm{\nabla c_0}_{L^2}, \norm{u_0}_{L^2})$.

\end{lemma}

\begin{proof}
We observe, by integrating both sides of
\eqref{eq:Chemotaxis-approximate}$_1$, that the total mass of $n$ is
preserved, \emph{i.e.} $\norm{n(t)}_1\equiv\norm{n_0}_1.$ It is also
obvious by applying maximal principle to
\eqref{eq:Chemotaxis-approximate}$_2$ that
$\norm{c}_{L^\infty(\mathbb R^3_T)}\leq\norm{c_0}_\infty,$ where
$\mathbb R^3_T:=\mathbb R^3\times [0,T).$ We also note here that $n$
and $c$ preserves nonnegativity of the initial data by the
condition $\kappa(0)=0$ and the parabolic comparison principle.\\

$\bullet$\,\textbf{Case {\bf (i)}\,\, in the Assumption
\ref{assume1-chi-kappa}}\,\, Here we consider the case that
$\alpha>1/6$. We separate the range of $\alpha$ into three parts,
that is, $1/6<\alpha\leq 1/3,$ $1/3<\alpha\leq 1$ and $\alpha>1$,
and treat each case individually. We start with the case that
$1/6<\alpha\leq 1/3$.\\
\\
(Case $1/6<\alpha\leq 1/3$):\,\, Multiplying
$(\ref{eq:Chemotaxis-approximate})_1$ with $(1+\log n)$ and
integrating it parts,
    $$
        \frac{d}{dt}\int_{\mathbb R^3} n\log n
        dx+\int_{\mathbb R^3}\nabla \log n\cdot \nabla(n+\varrho)^{1+\alpha}dx=\int_{\mathbb R^3}\nabla
        n\cdot\left(\chi(c)\nabla c\right)dx.
    $$
Since $\nabla\log n\cdot \nabla n=4\abs{\nabla n^{1/2}}^2,$ we have
    \begin{align*}
        &\int_{\mathbb R^3}\nabla \log n\cdot
        \nabla(n+\varrho)^{1+\alpha}dx=\int_{\mathbb R^3}\nabla\log n\cdot (1+\alpha)(n+\varrho)^\alpha\nabla n~ dx\\
        &\qquad\ge\int_{\mathbb R^3}\nabla \log n\cdot (1+\alpha)n^\alpha\nabla
        n~dx=\frac{4}{1+\alpha}\norm{\nabla n^{\frac{1+\alpha}{2}}}_2^2.
    \end{align*}
Taking into account $1/6<\alpha\leq 1/3,$ which admits
$(1-\alpha)/2>0,$ we obtain
    \begin{align*}
        &\frac{d}{dt}\int_{\mathbb R^3} n\log n
        dx+\frac{4}{1+\alpha}\left\|\nabla
        n^{\frac{1+\alpha}{2}}\right\|_2^2\leq\int_{\mathbb R^3}\nabla
        n\cdot\left(\chi(c)\nabla c\right)~dx\\
        &\qquad\leq \frac{2\overline\chi}{1+\alpha}\int_{\mathbb
        R^3}\abs{\nabla n^\frac{1+\alpha}{2}}\bke{n^\frac{1-\alpha}{2}\abs{\nabla
        c}}~dx,
    \end{align*}
where $\overline\chi:=\sup_{\mathbb R^3_T}\abs{\chi\bke{c(\cdot)}}.$
Applying Young's inequality, we observe
    $$
        \int_{\mathbb
        R^3}\abs{\nabla n^\frac{1+\alpha}{2}}\bke{n^\frac{1-\alpha}{2}\abs{\nabla
        c}}~dx\leq\epsilon\int_{\mathbb R^3}\norm{\nabla
        n^\frac{1+\alpha}{2}}^2_2~dx+C(\epsilon)\int_{\mathbb R^3}
        n^{1-\alpha}\abs{\nabla c}^2~dx.
    $$
Combining above estimate and choosing sufficiently small
$\epsilon>0$, we have
    \begin{equation}\label{eq:the_above}
        \frac{d}{dt}\int_{\mathbb R^3 }n\log n dx+C'\norm{\nabla
        n^\frac{1+\alpha}{2}}_2^2\leq C\int_{\mathbb
        R^3}n^{1-\alpha}\abs{\nabla c}^2~dx.
    \end{equation}
Reminding that $1-3\alpha\geq 0$ (since $\alpha\leq 1/3$) and that
$c\in L^\infty(\mathbb R^3_T),$ the term $\int_{\mathbb R^3}
n^{1-\alpha}\abs{\nabla c}^2~dx$ is estimated as
    \begin{align*}
        \int_{\mathbb R^3}n^{1-\alpha}\abs{\nabla c}^2~dx&=\int_{\mathbb
        R^3} n^{1-\alpha}\nabla c\cdot\nabla c~ dx\\
        &\leq C\bke{\int_{\mathbb R^3}\abs{\nabla
        n^{1-\alpha}}\abs{\nabla c}~dx+\int_{\mathbb R^3}
        n^{1-\alpha}\abs{\Delta c}~dx}\\
        &\leq \int_{\mathbb R^3}(\epsilon\abs{\nabla
        n^\frac{1+\alpha}{2}}^2+C(\epsilon)n^{1-3\alpha}\abs{\nabla
        c}^2) dx+C\int_{\mathbb R^3} n^{1-\alpha}\abs{\Delta c}~dx.
    \end{align*}
Choosing small $\epsilon>0$ in the above, we estimate
\eqref{eq:the_above} as follows:
    \begin{equation}\label{eq:first}
        \frac{d}{dt}\int_{\mathbb R^3}n\log n~dx+C'_1\norm{\nabla
        n^\frac{1+\alpha}{2}}_2^2\leq C_1\int_{\mathbb R^3}n^{1-3\alpha}
        \abs{\nabla c}^2+\abs{n^{1-\alpha}}\abs{\Delta c}~dx
    \end{equation}
for some $C_1>0$ and $C'_1>0.$ Next, Multiplying
\eqref{eq:Chemotaxis-approximate}$_1$ with $n^\alpha$ and
integrating it by part gives
    \begin{align*}
        \frac{1}{1+\alpha}\frac{d}{dt}\norm{n}_{1+\alpha}^{1+\alpha}+\frac{4\alpha(1+\alpha)}{\bke{1+2\alpha}^2}
        \norm{\nabla n^\frac{1+2\alpha}{2}}_2^2&=\int_{\mathbb R^3}\nabla
        n^\alpha\cdot n\chi(c)\nabla c~dx\\
        &\leq \frac{2\alpha\overline\chi}{1+\alpha}\int_{\mathbb
        R^3}\abs{\nabla n^\frac{1+2\alpha}{2}}\bke{n^\frac{1}{2}\abs{\nabla
        c}}~dx.
    \end{align*}
Applying Young's inequality to the integrand in the right-hand side
of the above, we have
    \begin{equation}\label{eq:n^alpha_test}
        \frac{d}{dt}\norm{n}_{1+\alpha}^{1+\alpha}+C'\norm{\nabla
        n^\frac{1+2\alpha}{2}}_2^2\leq C\int_{\mathbb R^3}n\abs{\nabla c}^2~dx
    \end{equation}
for some $C'>0.$ Since the integral term of the right-hand side in
\eqref{eq:n^alpha_test} is estimated as
    \begin{align*}
        \int_{\mathbb R^3}n\abs{\nabla c}^2~dx&=\int_{\mathbb R^3}n\nabla c\cdot\nabla
        c~dx\\
        &\leq C\bke{\int_{\mathbb R^3}\abs{\nabla n}\abs{\nabla c}~dx+\int_{\mathbb R^3}n\abs{\Delta
        c}~dx}\\
        &\leq C\bke{\int_{\mathbb R^3}n^\frac{1-2\alpha}{2}\abs{\nabla n^\frac{1+2\alpha}{2}}\abs{\nabla
        c}~dx+\int_{\mathbb R^3}n\abs{\Delta c}~dx},
    \end{align*}
it follows from Young's inequality that
    \begin{equation}\label{eq:second}
        \frac{d}{dt}\norm{n}_{1+\alpha}^{1+\alpha}+C'_2\norm{\nabla
        n^\frac{1+2\alpha}{2}}_2^2\leq C_2\bke{\int_{\mathbb R^3}n^{1-2\alpha}\abs{\nabla c}^2
        +n\abs{\Delta c}~dx}
    \end{equation}
for some $C_2>0$ and $C'_2>0.$ On the other hand, multiplying
\eqref{eq:Chemotaxis-approximate}$_2$ with $-\Delta c$ and using
integration by part, we get
    $$
        \frac{1}{2}\frac{d}{dt}\norm{\nabla
        c}_2^2+\norm{\Delta c}_2^2\leq\int_{\mathbb R^3}\bke{u\cdot\nabla
        c}\Delta c~dx+\overline\kappa\int_{\mathbb R^3}n\abs{\Delta c}~dx
    $$
where $\overline\kappa:=\max_{\mathbb R^3_T}\abs{\kappa(c(\cdot))}.$
Since the term $\int_{\mathbb R^3}\bke{u\cdot\nabla c}\Delta c~dx$
in the above is estimated as
    \begin{align*}
        \int_{\mathbb R^3}\bke{u\cdot\nabla c}\Delta c~dx&=\sum_{1\leq i,j\leq
        3}\int_{\mathbb R^3}~u_ic_{x_i}c_{x_jx_j}dx=-\sum_{1\leq i,j\leq
        3}\int_{\mathbb R^3}u_{i,x_j}c_{x_i}c_{x_j}~dx\\
        &=\sum_{1\leq i,j\leq 3}\int_{\mathbb
        R^3}u_{i,x_j}cc_{x_jx_i}~dx\leq C\norm{\nabla u}_2\norm{\nabla^2
        c}_2
    \end{align*}
where $u_i$ and $c_{x_i}$ denote the $i$-th component of $u$ and
$\frac{\partial c}{\partial x_i}$ respectively, we have
    \begin{equation}\label{eq:third}
        \frac{d}{dt}\norm{\nabla c}_2^2+\norm{\Delta c}_2^2\leq
        C_3\bke{\norm{\nabla u}_2\norm{\Delta c}_2+\int_{\mathbb R^3}n\abs{\Delta
        c}dx}
    \end{equation}
for some $C_3>0.$ For the fluid equation, energy estimate gives
    \begin{equation}\label{eq:fourth}
        \frac{d}{dt}\norm{u}_2^2+2\norm{\nabla
        u}_2^2\leq C_4\int_{\mathbb R^3}n\abs{u}~dx
    \end{equation}
for some $C_4>0.$ Finally, to bound $\int n\abs{\log n}$ in
(\ref{eq:first}), we recall that (see e.g. \cite{Liu-Lorz} and
\cite{CKL})
    \begin{equation}\label{CK-Nov27-10}
        \int_{\mathbb R^3} n|\log n|dx\leq \int_{\mathbb R^3} n\log
        n dx+2\int_{\mathbb R^3}\langle x\rangle ndx+C,
    \end{equation}
where $\langle x\rangle=(1+|x|^2)^{1/2}$ and (see e.g. $(18)$ of
\cite{CKK})
    \begin{align}\label{eq:fifth}
        \frac{d}{dt}\int_{\mathbb R^3} \langle x\rangle n
        dx&=\int_{\mathbb R^3}nu\cdot\nabla\langle x
        \rangle +\int_{\mathbb R^3}(n+\varrho)^{1+\alpha}\Delta\langle x\rangle
        +\int_{\mathbb R^3}\nabla \langle x\rangle\cdot n\chi(c)\nabla c\nonumber\\
        &\leq
        C\left(1+\|u\|_2^2+\|\nabla c\|_2^2\right)+\left(C(\epsilon)+\epsilon\|\nabla
        n^\frac{1+\alpha}{2}\|_2^2\right).
    \end{align}
 Summing up
\eqref{eq:first}$\sim$\eqref{eq:fifth} with sufficiently small
$\epsilon>0$, we have
    \begin{align}\label{eq:med_sum}
        \frac{d}{dt}&\left(\int_{\mathbb R^3}n(\log n+2\langle x\rangle)~dx+\norm{n}_{1+\alpha}^{1+\alpha}+\norm{\nabla c}_2^2+\norm{u}_2^2
        \right)\\
        \nonumber &+C'_5 \left(\norm{\nabla n^\frac{1+\alpha}{2}}_2^2+\norm{\nabla n^\frac{1+2\alpha}{2}}_2^2
        +\norm{\Delta c}_2^2+\norm{\nabla
        u}_2^2\right)\\
        \nonumber \leq C_5&\left(\int_{\mathbb R^3}n^{1-3\alpha}\abs{\nabla c}^2dx
        +\int_{\mathbb R^3}n^{1-2\alpha}\abs{\nabla c}^2dx
        +\int_{\mathbb R^3}n^{1-\alpha}\abs{\Delta c}dx\right.\\
        \nonumber&+\int_{\mathbb R^3} n\abs{\Delta
        c}dx +\left. \norm{\nabla u}_2\norm{\Delta c}_2+\int_{\mathbb R^3}n\abs{u}
        dx\right).\\
        \nonumber:=C_5&\bke{\rm{I}+\rm{II}+\rm{III}+\rm{IV}+\rm{V}+\rm{VI}}.
    \end{align}
for some $C'_5>0$ and $C_5>0.$ Due to $0\leq 1-3\alpha <2/3$ and $0<
1-2\alpha <2/3$, we note that
    \begin{equation*}
        \rm{I}\leq\left\{
             \begin{array}{cl}
               \int_{\mathbb R^3} \bke{C(\epsilon_1)+\epsilon_1 n^\frac{2}{3}} \abs{\nabla c}^2~dx, & \mbox{if}\,\,1/6<\alpha<1/3; \\
               \norm{\nabla c}_2^2, & \mbox{if}\,\,\alpha=1/3,
             \end{array}
           \right.
    \end{equation*}
and
    \begin{equation*}
        \mathrm{II}\leq\int_{\mathbb R^3}\bke{C(\epsilon_2)+\epsilon_2n^\frac{2}{3}}\abs{\nabla c}^2~dx,
    \end{equation*}
where the positive constants $\epsilon_1$ and $\epsilon_2$ will be
determined later. Hence, it follows from H$\mathrm{\ddot{o}}$lder
inequality and Sobolev embedding that
    \begin{equation}\label{eq:A1}
        \rm{I}\leq\left\{
             \begin{array}{cl}
               C(\epsilon_1)\norm{\nabla c}_2^2+\epsilon_1\norm{n_0}_1^\frac{2}{3}\norm{\Delta c}_2^2, &
               \mbox{if}\,\, 1/6<\alpha<1/3; \\
               \norm{\nabla c}_2^2, & \mbox{if}\,\,\alpha=1/3,
             \end{array}
           \right.
    \end{equation}
and
    \begin{equation}\label{eq:A2}
        \mathrm{II}\leq C(\epsilon_2)\norm{\nabla c}_2^2+\epsilon_2\norm{n_0}_1^\frac{2}{3}\norm{\Delta
        c}_2^2.
    \end{equation}

To estimate III, we apply H$\mathrm{\ddot{o}}$lder and Young
inequality to have $\mathrm{III}\leq
C(\epsilon)\norm{n}_{2-\alpha}^{2-\alpha}+\epsilon\norm{\Delta
c}_2^2.$ Noting that $\frac{4}{3}\leq\frac{6-6\alpha}{2+3\alpha}<2$
(which is obtained from the condition $1/6<\alpha\leq 1/3$), we
obtain, from Young's inequality and Gagliardo-Nierenberg inequality,
the following :
    $$
        \norm{n}_{2-\alpha}^{2-\alpha}\leq
        C\norm{n_0}_1^\frac{1+4\alpha}{2+3\alpha}\norm{\nabla
        n^\frac{1+\alpha}{2}}_2^\frac{6-6\alpha}{2+3\alpha}\leq
        C(\epsilon')\norm{n_0}_1^\frac{1+4\alpha}{2+3\alpha}+\epsilon'\norm{\nabla
        n^\frac{1+\alpha}{2}}_2^2.
    $$
Hence by choosing sufficient $\epsilon'>0$, we have
    \begin{equation}\label{eq:A3}
        \mathrm{III}\leq
        C(\epsilon_3)+\epsilon_3\norm{\nabla n^\frac{1+\alpha}{2}}_2^2+\epsilon_3\norm{\Delta
        c}_2^2
    \end{equation}
where $\epsilon_3>0$ will be specified later. We estimate \rm{IV}
via a calculation similar to the above as follows.
H$\rm{\ddot{o}}$lder and Young inequality give $\mathrm{IV}\leq
C(\epsilon)\norm{n}_2^2+\epsilon\norm{\Delta c}_2^2.$ Since
$\frac{3}{2}\leq\frac{6}{2+6\alpha}<2$ via $1/6<\alpha\leq 1/3,$ it
follows from Young and Gagliardo-Nierenberg inequality that
    $$
        \norm{n}_2^2\leq
        C\norm{n_0}_1^\frac{1+6\alpha}{2+6\alpha}\norm{\nabla
        n^\frac{1+2\alpha}{2}}_2^\frac{6}{2+6\alpha}\leq
        C(\epsilon')\norm{n_0}_1^\frac{1+6\alpha}{2+3\alpha}+\epsilon'\norm{\nabla
        n^\frac{1+2\alpha}{2}}_2^2.
    $$
Therefore, we have
    \begin{equation}\label{eq:A4}
        \mathrm{IV}\leq C(\epsilon_4)+\epsilon_4\norm{\nabla
        n^\frac{1+2\alpha}{2}}_2^2+\epsilon_4\norm{\Delta c}_2^2.
    \end{equation}
It is straightforward that V can be estimated as
    \begin{equation}\label{eq:A5}
        \mathrm{V}\leq C(\epsilon_5)\norm{\nabla
        u}_2^2+\epsilon_5\norm{\Delta c}_2^2.
    \end{equation}
The positive constants $\epsilon_4$ in \eqref{eq:A4} and
$\epsilon_5$ in \eqref{eq:A5} will also be specified later. Finally,
we estimate VI as follows. Via H$\rm{\ddot{o}}$lder and Young
inequality and Sobolev embedding, we obtain
    $$
        \mathrm{VI}\leq C\norm{n}_\frac{6}{5}\norm{u}_{6}\leq
        C(\epsilon)\norm{n}_\frac{6}{5}+\epsilon\norm{\nabla
        u}_2^2.
    $$
Since $1<\frac{6}{5}<3+6\alpha$ and $0<\frac{2}{2+6\alpha}<2$,
$\norm{n}_\frac{6}{5}^2$  is estimated, via Gagliardo-Nierenberg
inequality and Young's inequality, as
    $$
        \norm{n}_\frac{6}{5}^2\leq\norm{n_0}_1^\frac{(5)(1+2\alpha)-2}{2+6\alpha}
        \norm{\nabla
        n^\frac{1+2\alpha}{2}}_2^\frac{2}{2+6\alpha}\leq
        C(\epsilon')+\epsilon'\norm{\nabla
        n^\frac{1+2\alpha}{2}}_2^2.
    $$
Thus, we conclude that
    \begin{equation}\label{eq:A6}
        \mathrm{VI}\leq C(\epsilon_6)+\epsilon_6\norm{\nabla
        n^\frac{1+2\alpha}{2}}_2^2+\epsilon_6\norm{\nabla
        u}_2^2.
    \end{equation}
Summing up \eqref{eq:A1}$\sim$\eqref{eq:A6} with choosing
sufficiently small $\epsilon_1,...,\epsilon_6>0$, we finally have
    \begin{align}\label{eq:med_sum_2}
        \frac{d}{dt}&\left(\int_{\mathbb R^3}n(\log n+2\langle x\rangle)~dx+\norm{n}_{1+\alpha}^{1+\alpha}+\norm{\nabla c}_2^2+\norm{u}_2^2
        \right)\\
        \nonumber &+C'_6 \left(\norm{\nabla n^\frac{1+\alpha}{2}}_2^2+\norm{\nabla n^\frac{1+2\alpha}{2}}_2^2
        +\norm{\Delta c}_2^2+\norm{\nabla
        u}_2^2\right)\\
        \nonumber \leq C_6&\bke{\norm{\nabla c}_2^2+\norm{\nabla u}_2^2+1}
    \end{align}
for some $C'_6>0$ and $C_6>0.$ To absorb the term $C_6\norm{\nabla
u}_2^2$ in the right-hand side of the above, we test $u$ to the
\eqref{eq:Chemotaxis-approximate}$_3$ and then multiply sufficiently
large constant $M>C_6$ to have
    \begin{equation}\label{eq:the_above_2}
        \frac{M}{2}\frac{d}{dt}\norm{u}_2^2+M\norm{\nabla u}_2^2\leq
        CM\int_{\mathbb R^3}n\abs{u}~dx.
    \end{equation}
As \eqref{eq:A6}, the integral in the right-hand side of the above
is estimated as
    $$
        \int_{\mathbb R^3}n\abs{u}~dx\leq
        C(\epsilon)+\epsilon\norm{\nabla
        n^\frac{1+2\alpha}{2}}_2^2+\epsilon\norm{\nabla u}_2^2.
    $$
Hence substituting it to \eqref{eq:the_above_2}, choosing
sufficiently small $\epsilon>0$ and then adding it to
\eqref{eq:med_sum_2}, we have
    \begin{align}\label{eq:final_sum}
        \frac{d}{dt}&\left(\int_{\mathbb R^3}n(\log n+2\langle x\rangle)~dx+\norm{n}_{1+\alpha}^{1+\alpha}+\norm{\nabla c}_2^2
        +\frac{M+2}{2}\norm{u}_2^2\right)\\
        \nonumber &+C' \left(\norm{\nabla n^\frac{1+\alpha}{2}}_2^2+\norm{\nabla n^\frac{1+2\alpha}{2}}_2^2
        +\norm{\Delta c}_2^2+\norm{\nabla u}_2^2\right)\\
        \nonumber \leq C&\bke{\norm{\nabla c}_2^2+1}.
    \end{align}
Integrating in time variables and combining estimate
\eqref{CK-Nov27-10}, we deduce \eqref{eq:E'+D<CE}.
\\
\\
(Case $1/3<\alpha\leq 1$):\,\, Performing similar calculations as in
the above case, we obtain
    \begin{equation}\label{eq:first_2}
        \frac{d}{dt}\int_{\mathbb R^3} n\log n~dx+C'_1\norm{\nabla
        n^\frac{1+\alpha}{2}}_2^2\leq C_1\int_{\mathbb R^3}n^{1-\alpha}\abs{\nabla
        c}^2~dx,
    \end{equation}
    \begin{align}\label{eq:second_2}
        \frac{d}{dt}\norm{n}_{1+\alpha}^{1+\alpha}&+C'_2\norm{\nabla
        n^\frac{1+2\alpha}{2}}_2^2 \leq C\int_{\mathbb R^3}n\abs{\nabla
        c}^2~dx\\
        &\leq C_{2,\epsilon}\bke{\int_{\mathbb R^3}n^{1-\alpha}\abs{\nabla c}^2
        +n\abs{\Delta c}~dx}+\epsilon\norm{\nabla n^\frac{1+\alpha}{2}}_2^2\nonumber,
    \end{align}
    \begin{equation}\label{eq:third_2}
        \frac{d}{dt}\norm{\nabla c}_2^2+\norm{\Delta c}_2^2\leq
        C_3\bke{\norm{\nabla u}_2\norm{\Delta c}_2+\int_{\mathbb R^3}~n\abs{\Delta
        c}dx},
    \end{equation}
    \begin{equation}\label{eq:fourth_2}
        \frac{d}{dt}\norm{u}_2^2+2\norm{\nabla
        u}_2^2\leq C_4\int_{\mathbb R^3}n\abs{u}~dx
    \end{equation}
and
    \begin{align}\label{eq:fifth_2}
        \frac{d}{dt}\int_{\mathbb R^3} \langle x\rangle n
        dx&=\int_{\mathbb R^3}nu\cdot\nabla\langle x
        \rangle +\int_{\mathbb R^3}(n+\varrho)^{1+\alpha}\Delta\langle x\rangle
        +\int_{\mathbb R^3}\nabla \langle x\rangle\cdot n\chi(c)\nabla c\nonumber\\
        &\leq
        C_5\left(1+\|u\|_2^2+\|\nabla c\|_2^2\right)+\left(C(\epsilon)+\epsilon\|\nabla
        n^\frac{1+\alpha}{2}\|_2^2\right).
    \end{align}
Summing up
\eqref{eq:first_2}$\sim$\eqref{eq:fifth_2} with sufficiently small
$\epsilon>0$, we have
    \begin{align}\label{eq:med_sum22}
        \frac{d}{dt}&\left(\int_{\mathbb R^3}n(\log n+2\langle x\rangle)~dx+\norm{n}_{1+\alpha}^{1+\alpha}+\norm{\nabla c}_2^2+\norm{u}_2^2
        \right)\\
        \nonumber &+C'_6 \left(\norm{\nabla n^\frac{1+\alpha}{2}}_2^2+\norm{\nabla n^\frac{1+2\alpha}{2}}_2^2
        +\norm{\Delta c}_2^2+\norm{\nabla
        u}_2^2\right)\\
        \nonumber \leq C_6&\left(\int_{\mathbb R^3}n^{1-\alpha}\abs{\nabla c}^2dx+\int_{\mathbb R^3} n\abs{\Delta
        c}dx + \norm{\nabla u}_2\norm{\Delta c}_2+\int_{\mathbb R^3}n\abs{u}
        dx\right)\\
        \nonumber=:C_6&\bke{\rm{I}+\rm{II}+\rm{III}+\rm{IV}}.
    \end{align}
for some $C'_6>0$ and $C_6>0.$ Since $1/3<\alpha\leq 1$, it is
direct that $0\leq 1-\alpha<2/3$ and $0<\frac{6}{2+6\alpha}<2,$ and
thus, we obtain the same estimations as those in the previous case
for \rm{I}$\sim$\rm{IV}. To be more precise, we have
    \begin{equation}\label{eq:B1}
        \rm{I}\leq\left\{
             \begin{array}{cl}
               C(\epsilon_1)\norm{\nabla c}_2^2+\epsilon_1\norm{n_0}_1^\frac{2}{3}\norm{\Delta c}_2^2, &
               \mbox{if}\,\,1/3<\alpha<1; \\
               \norm{\nabla c}_2^2, &  \mbox{if}\,\,\alpha=1,
             \end{array}
           \right.
    \end{equation}
    \begin{equation}\label{eq:B2}
        \mathrm{II}\leq C(\epsilon_2)+\epsilon_2\norm{\nabla
        n^\frac{1+2\alpha}{2}}_2^2+\epsilon_2\norm{\Delta c}_2^2,
    \end{equation}
    \begin{equation}\label{eq:B3}
        \mathrm{III}\leq C(\epsilon_3)\norm{\nabla
        u}_2^2+\epsilon_3\norm{\Delta c}_2^2
    \end{equation}
and
    \begin{equation}\label{eq:B4}
        \mathrm{IV}\leq C(\epsilon_4)+\epsilon_4\norm{\nabla
        n^\frac{1+2\alpha}{2}}_2^2+\epsilon_4\norm{\nabla
        u}_2^2.
    \end{equation}
Following the same procedure as that in the previous case, we
conclude that
    \begin{align}\label{eq:final_sum2}
        \frac{d}{dt}&\left(\int_{\mathbb R^3}n(\log n+2\langle x\rangle)~dx+\norm{n}_{1+\alpha}^{1+\alpha}+\norm{\nabla c}_2^2
         +\frac{M+2}{2}\norm{u}_2^2\right)\\
        \nonumber &+C' \left(\norm{\nabla n^\frac{1+\alpha}{2}}_2^2+\norm{\nabla n^\frac{1+2\alpha}{2}}_2^2
        +\norm{\Delta c}_2^2+\norm{\nabla u}_2^2\right)\\
        \nonumber \leq C&\bke{\norm{\nabla c}_2^2+1},
    \end{align}
for some $M,~C,~C'>0,$ which gives the result after time integration
together with estimate \eqref{CK-Nov27-10}.
\\
\\
(Case $\alpha>1$):\,\, As in the previous cases, we obtain
    \begin{align}\label{eq:first_3}
        \frac{d}{dt}\int_{\mathbb R^3} n\log n
            dx+C_1'\left\|\nabla
            n^{\frac{1+\alpha}{2}}\right\|_2^2\leq&\int_{\mathbb R^3}\nabla
            n\cdot\left(\chi(c)\nabla c\right)~dx\\
        \nonumber\leq&C_1\int_{\mathbb R^3}n\abs{\nabla c}^2+n\abs{\Delta c}~dx,
    \end{align}
    \begin{equation}\label{eq:second_3}
        \frac{d}{dt}\norm{n}_{1+\alpha}^{1+\alpha}+C_2'\norm{\nabla
        n^\frac{1+2\alpha}{2}}_2^2\leq C_2\int_{\mathbb R^3}n\abs{\nabla
        c}^2~dx\quad\mbox{(see (\ref{eq:n^alpha_test}))},
    \end{equation}
    \begin{equation}\label{eq:third_3}
        \frac{d}{dt}\norm{\nabla c}_2^2+\norm{\Delta c}_2^2\leq
        C_3\bke{\norm{\nabla u}_2\norm{\Delta c}_2+\int_{\mathbb R^3}~n\abs{\Delta
        c}dx},
    \end{equation}
    \begin{equation}\label{eq:fourth_3}
        \frac{d}{dt}\norm{u}_2^2+2\norm{\nabla
        u}_2^2\leq C_4\int_{\mathbb
        R^3}n\abs{u}~dx
        \end{equation}
and
    \begin{align}\label{eq:fifth_3}
        \frac{d}{dt}\int_{\mathbb R^3} \langle x\rangle n
        dx&=\int_{\mathbb R^3}nu\cdot\nabla\langle x
        \rangle +\int_{\mathbb R^3}(n+\varrho)^{1+\alpha}\Delta\langle x\rangle
        +\int_{\mathbb R^3}\nabla \langle x\rangle\cdot n\chi(c)\nabla c\nonumber\\
        &\leq
        C_5\left(1+\|u\|_2^2+\|\nabla c\|_2^2\right)+\left(C(\epsilon)+\epsilon\|\nabla
        n^\frac{1+\alpha}{2}\|_2^2\right)
    \end{align}
 for some constants $C'_1, C'_2, C_1\sim C_5>0.$ Summing up
\eqref{eq:first_3}$\sim$\eqref{eq:fifth_3} with sufficiently small
$\epsilon>0$, we have
    \begin{align}\label{eq:med_sum2}
        \frac{d}{dt}&\left(\int_{\mathbb R^3}n(\log n+2\langle x\rangle)~dx+\norm{n}_{1+\alpha}^{1+\alpha}+\norm{\nabla c}_2^2+\norm{u}_2^2
        \right)\\
        \nonumber &+C'_5 \left(\norm{\nabla n^\frac{1+\alpha}{2}}_2^2+\norm{\nabla n^\frac{1+2\alpha}{2}}_2^2
        +\norm{\Delta c}_2^2+\norm{\nabla
        u}_2^2\right)\\
        \nonumber \leq C_5&\left(\int_{\mathbb R^3}n\abs{\nabla c}^2dx+\int_{\mathbb R^3} n\abs{\Delta
        c}dx + \norm{\nabla u}_2\norm{\Delta c}_2+\int_{\mathbb R^3}n\abs{u}
        dx\right)\\
        \nonumber:=C_5&\bke{\rm{I}+\rm{II}+\rm{III}+\rm{IV}}
    \end{align}
for some $C'_5>0$ and $C_5>0.$ We now estimate \rm{I} as follows.
Taking into account that $\frac{1+\alpha}{2}>1$ via $\alpha>1$, we
have $n\leq C(\epsilon)+ \epsilon n^{\frac{1+\alpha}{2}}$ due to
Young's inequality and, therefore, we get
    \begin{align*}
        \mathrm{I}\leq C(\epsilon)\norm{\nabla c}_2^2+\epsilon\int_{\mathbb R^3}
            n^\frac{1+\alpha}{2}\abs{\nabla c}^2~dx.
    \end{align*}
The second term in the right-hand side of the above is estimated by
    \begin{align*}
        \int_{\mathbb R^3} n^\frac{1+\alpha}{2}\nabla c\cdot\nabla
            c~dx \leq&C\bke{\int_{\mathbb R^3}\nabla n^\frac{1+\alpha}{2}\cdot\nabla c+n^\frac{1+\alpha}{2}\Delta c
            ~dx}\\
        \leq&C\bke{\norm{\nabla n^\frac{1+\alpha}{2}}_2^2+\norm{\nabla c}_2^2+\norm{n}_{1+\alpha}^{1+\alpha}+\norm{\Delta
            c}_2^2}.
    \end{align*}
We also note, due to Gagliardo-Nierenberg inequality and Young's
inequality that
\[
\norm{n}_{1+\alpha}^{1+\alpha}\leq
C\norm{n_0}_1^\frac{2+2\alpha}{2+3\alpha}\norm{\nabla
n^\frac{1+\alpha}{2}}_2^\frac{6\alpha}{2+3\alpha}\leq
C'+\norm{\nabla n^\frac{1+\alpha}{2}}_2^2,
\]
where we used that $\frac{6\alpha}{2+3\alpha}<2$. Combining
estimates, we conclude that
    \begin{equation}\label{eq:C1}
        \mathrm{I}\leq C(\epsilon_1)\bke{\norm{\nabla
        c}_2^2+1}+\epsilon_1\bke{\norm{\nabla n^\frac{1+\alpha}{2}}_2^2+\norm{\Delta
        c}_2^2}.
    \end{equation}
We observe that the other terms $\rm{II}\sim\rm{IV}$ are estimated
exactly the same as those in previous case. More precisely, we have
    \begin{equation}\label{eq:C2}
        \mathrm{II}\leq C(\epsilon_2)+\epsilon_2\norm{\nabla
        n^\frac{1+2\alpha}{2}}_2^2+\epsilon_2\norm{\Delta c}_2^2,
    \end{equation}
    \begin{equation}\label{eq:C3}
        \mathrm{III}\leq C(\epsilon_3)\norm{\nabla
        u}_2^2+\epsilon_3\norm{\Delta c}_2^2,
    \end{equation}
    \begin{equation}\label{eq:C4}
        \mathrm{IV}\leq C(\epsilon_4)+\epsilon_4\norm{\nabla
        n^\frac{1+2\alpha}{2}}_2^2+\epsilon_4\norm{\nabla
        u}_2^2.
    \end{equation}
Following similar procedures as before, we conclude that
    \begin{align}\label{eq:final_sum3}
        \frac{d}{dt}&\left(\int_{\mathbb R^3}n\log n~dx+\norm{n}_{1+\alpha}^{1+\alpha}+\norm{\nabla c}_2^2
        +\frac{M+2}{2}\norm{u}_2^2\right)\\
        \nonumber &+C' \left(\norm{\nabla n^\frac{1+\alpha}{2}}_2^2+\norm{\nabla n^\frac{1+2\alpha}{2}}_2^2
        +\norm{\Delta c}_2^2+\norm{\nabla
        u}_2^2\right)\\
        \nonumber \leq C&\bke{\norm{\nabla c}_2^2+1}
    \end{align}
for some $M,~C,~C'>0,$ which implies \eqref{eq:E'+D<CE}.
Since its verifications are just repetition of previous computations, the details are omitted.\\

$\bullet$\,\textbf{Case {\bf (ii)} and {\bf(iii)} in the Assumption
\ref{assume1-chi-kappa}}\,\, Here, we only deal with only the case
(iii), since the case (ii) can be proved in almost the same manner.
Due to the previous result, since the case $\alpha>1/6$ is direct,
it suffices to consider the case $0<\alpha\leq 1/6$. Multiplying
\eqref{eq:Chemotaxis-approximate}$_2$ with $-\Delta c$ and
integrating it by part gives
    $$
        \frac{1}{2}\frac{d}{dt}\norm{\nabla
        c}_2^2+\norm{\Delta c}_2^2\leq\int_{\mathbb R^3}\bke{u\cdot\nabla
        c}\Delta c~dx+\int_{\mathbb R^3}\kappa(c)n\Delta c~dx.
    $$
Since the last integral of the above is estimated as
    \begin{align*}
        \int_{\mathbb R^3}\kappa(c)n\Delta c~dx&=-\int_{\mathbb
        R^3}\bke{\kappa'(c)n\nabla c+\kappa(c)\nabla n}\nabla
        c~dx\\
        &\leq-\kappa_0\int_{\mathbb R^3}n\abs{\nabla
        c}^2+\overline{\kappa}\int_{\mathbb R^3}\abs{\nabla n}\abs{\nabla
        c}~dx,
    \end{align*}
we have
    \begin{align}\label{eq:case2_1}
        \nonumber\frac{d}{dt}\norm{\nabla c}_2^2&+\norm{\Delta c}_2^2+2\kappa_0
        \int_{\mathbb R^3}n\abs{\nabla c}^2~dx\\
        &\leq
        C\bke{\norm{\nabla u}_2\norm{\Delta c}_2+\int_{\mathbb R^3}\abs{\nabla
        n^\frac{1+\alpha}{2}}
        \bke{n^\frac{1-\alpha}{2}\abs{\nabla c}}dx}.
    \end{align}
Multiplying $(\ref{eq:Chemotaxis-approximate})_1$ with $(1+\log n)$
and integrating it by parts, we get
    \begin{align}\label{eq:case2_2}
        \nonumber&\frac{d}{dt}\int_{\mathbb R^3} n\log n
        dx+\frac{4}{1+\alpha}\left\|\nabla
        n^{\frac{1+\alpha}{2}}\right\|_2^2\leq\int_{\mathbb R^3}\nabla
        n\cdot\left(\chi(c)\nabla c\right)~dx\\
        &\qquad\leq \frac{2\overline\chi}{1+\alpha}\int_{\mathbb
        R^3}\abs{\nabla n^\frac{1+\alpha}{2}}\bke{n^\frac{1-\alpha}{2}\abs{\nabla
        c}}~dx.
    \end{align}
Since the integral in the last inequality of \eqref{eq:case2_1} and
\eqref{eq:case2_2} is estimated, due to Young's inequlity, as
    $$
        \int_{\mathbb R^3}\abs{\nabla n^\frac{1+\alpha}{2}}\bke{n^\frac{1-\alpha}{2}\abs{\nabla c}}
        \leq \epsilon\norm{\nabla
        n^\frac{1+\alpha}{2}}_{2}^{2}+C(\epsilon)\int_{\mathbb
        R^3}\abs{n^\frac{1-\alpha}{2}\nabla c}^2,
    $$
we obtain by choosing a sufficiently small $\epsilon$
    \begin{align}\label{eq:case2_3}
        \nonumber\frac{d}{dt}\left[\int_{\mathbb R^3}n\log ndx+\norm{\nabla c}_2^2
        \right]&+\left[\norm{\nabla n^\frac{1+\alpha}{2}}_2^2+\norm{\Delta c}_{2}^{2}+2\kappa_0\int_{\mathbb R^3}
        n\abs{\nabla c}^2dx\right]\\
        &\leq C\bke{\norm{\nabla u}_{2}\norm{\Delta c}_{2}+\int_{\mathbb
        R^3}n^{1-\alpha}\abs{\nabla c}^2 dx}.
    \end{align}
Multiplying \eqref{eq:Chemotaxis-approximate}$_1$ with $n^\alpha$
and integrating it by parts gives
    \begin{align*}
        \frac{1}{1+\alpha}\norm{n}_{1+\alpha}^{1+\alpha}+\frac{4\alpha(1+\alpha)}{\bke{1+2\alpha}^2}
        \norm{\nabla n^\frac{1+2\alpha}{2}}_2^2&=\int_{\mathbb R^3}\nabla
        n^\alpha\cdot n\chi(c)\nabla c~dx\\
        &\leq \frac{2\alpha\overline\chi}{1+\alpha}\int_{\mathbb
        R^3}\abs{\nabla n^\frac{1+2\alpha}{2}}\bke{n^\frac{1}{2}\abs{\nabla
        c}}~dx.
    \end{align*}
Young's inequality implies
    $$
        \frac{1}{1+\alpha}\frac{d}{dt}\norm{n}_{1+\alpha}^{1+\alpha}+\norm{\nabla
        n^\frac{1+2\alpha}{2}}_2^2\leq C_0\int_{\mathbb R^3}n\abs{\nabla
        c}^2 dx
    $$
for some $C_0>0.$ Now we choose sufficiently small $\epsilon_0>0$
satisfying $\epsilon_0C_0<2\kappa_0.$ Multiplying it by the both
sides of the above gives
    \begin{align}\label{eq:case2_4}
        \frac{\epsilon_0}{1+\alpha}\frac{d}{dt}\norm{n}_{1+\alpha}^{1+\alpha}+\epsilon_0\norm{\nabla
        n^\frac{1+2\alpha}{2}}_2^2\leq \epsilon_0C_0\int_{\mathbb
        R^3}n\abs{\nabla c}^2 dx.
    \end{align}
Summing up \eqref{eq:case2_3} and \eqref{eq:case2_4}, we have
    \begin{align*}
        &\frac{d}{dt}\left[\int_{\mathbb R^3}n\log n+\norm{\nabla
        c}_2^2+\frac{\epsilon_0}{1+\alpha}\norm{n}_{1+\alpha}^{1+\alpha}\right]\\
        &\qquad\qquad\quad+\left[\norm{\nabla n^\frac{1+\alpha}{2}}_{2}^{2}+\epsilon_0\norm{\nabla
        n^\frac{1+2\alpha}{2}}_{2}^{2}+\norm{\Delta c}_{2}^{2}+\delta_0\int_{\mathbb R^3}n\abs{\nabla c}^2
        dx\right]\\
        \leq &C\bke{\norm{\nabla u}_{2}\norm{\Delta c}_{2}+\int_{\mathbb
        R^3}n^{1-\alpha}\abs{\nabla c}^2 dx},
    \end{align*}
where $\delta_0:=2\kappa_0-\epsilon_0C_0>0.$ Therefore, applying
Young's inequality $n^{1-\alpha}\leq C(\epsilon)+\epsilon n$
        with sufficiently small $\epsilon>0$, we obtain
    \begin{align*}
        &\frac{d}{dt}\left[\int_{\mathbb R^3}n\log ndx+\norm{\nabla c}_2^2+\frac{\epsilon_0}{1+\alpha}
        \norm{n}_{1+\alpha}^{1+\alpha}\right]\\
        &\qquad\qquad\quad+\left[\norm{\nabla n^\frac{1+\alpha}{2}}_{2}^{2}+\epsilon_0\norm{\nabla n^\frac{1+2\alpha}{2}}_{2}^{2}
        +\norm{\Delta c}_{2}^{2}+\frac{\delta_0}{2}\int_{\mathbb R^3} n\abs{\nabla c}^2\right]\\
        \leq &C\bke{\norm{\nabla c}_{2}^{2}+\norm{\nabla u}_{2}^{2}}.
    \end{align*}
Finally, absorbing the term $\norm{\nabla u}_2^2$ as we have done
previously (see \eqref{eq:the_above_2}) and combining estimate
\eqref{eq:fifth}, we deduce \eqref{eq:E'+D<CE}. This completes the
proof.
\end{proof}

\section{Bounded weak solutions}

In this section, we provide some uniform estimates of solutions to
the system (\ref{eq:Chemotaxis-approximate}) with $\tau=0$, which is
used in Section 4 to construct the bounded weak solution of the
equation \eqref{eq:Chemotaxis} and prove Theorem
\ref{thm:3D-L^infinite bound}.
To be more precise, our result is read as follows:
\begin{lemma}\label{lemma:bdd-weak}
Let $\tau=0$ and $~T>0$ be given. Suppose that $(n,c,u)$ is a
classical solution for the system
(\ref{eq:Chemotaxis-approximate})$_\varrho,$ $\varrho\in(0,1)$ with
the smooth initial datum $(n_{0\varrho},c_{0\varrho},u_{0\varrho})$
which satisfies the initial conditions \eqref{initial-weak} and
\eqref{initial-boundweak}. Assume further that $\chi,\kappa$ and
$\alpha$ satisfies Assumption \ref{assume2-chi-kappa}. Then, for any
$0<t\le T$
%
\begin{equation}\label{CKK100-Dec14}
n\in L^{\infty}(0,T;L^p(\mathbb R^3)), \qquad 1\le p\leq \infty,
\end{equation}
\begin{equation}\label{CKK150-Dec14}
\nabla n^{\frac{p+\alpha}{2}}\in L^{2}(0,T;L^2(\mathbb R^3)), \qquad
1\le p< \infty,
\end{equation}
\begin{equation}\label{CKK200-Dec14}
c, u\in L^\infty(0,T;W^{1,q}(\mathbb R^3))\cap
L^q(0,T;W^{2,q}(\mathbb R^3)), \qquad 2\le q<\infty,
\end{equation}
\begin{equation}\label{CKK300-Dec14}
c_t, u_t\in L^q(0,T;L^q(\mathbb R^3)), \qquad 2\le q<\infty.
\end{equation}
\end{lemma}
\begin{proof}
Since Assumption \ref{assume2-chi-kappa} is stronger than Assumption
\ref{assume1-chi-kappa}, it is automatic, due to Lemma
\ref{lemma:weak}, that
    \begin{equation}\label{eq:from_lemma}
        n\in
        L^\infty(0,T;L^p(\mathbb R^3))\quad\mathrm{and}\quad\nabla
        n^{\frac{p+\alpha}{2}}\in L^2(0,T;L^2(\mathbb R^3)),\quad 1\leq
        p\leq 1+\alpha.
    \end{equation}
In what follows, we complete the proof by showing that for each
$\alpha>1/8,$ if $n$ satisfies \eqref{eq:from_lemma} then
\eqref{CKK100-Dec14} $\sim$ \eqref{CKK300-Dec14} hold. We remark
that the first case (i) in the Assumption \ref{assume2-chi-kappa},
i.e. $\alpha>1/6$, cannot be relaxed as $\alpha>1/8$, since
\eqref{eq:from_lemma} is not available in case only $\alpha>1/8$ is
assumed.

First, we consider $L^p$ estimate of
\eqref{eq:Chemotaxis-approximate}$_1$
    $$
        \frac{1}{p}\frac{d}{dt}\norm{n}_{p}^{p}
        +\int_{\mathbb R^3}\nabla n^{p-1}\cdot \nabla(n+\varrho)^pdx=-\int_{\mathbb
        R^3}n^{p-1}\nabla\cdot\bke{\chi(c)n\nabla c}~dx.
    $$
Since $\nabla n^{p-1}\cdot \nabla n=(4(p-1)/p^2)\abs{\nabla
n^{p/2}}^2\geq 0,$ we have
    $$
        \frac{1}{p}\frac{d}{dt}\norm{n}_p^p+\frac{4(p-1)(1+\alpha)}{(p+\alpha)^2}\norm{\nabla
        n^\frac{p+\alpha}{2}}_2^2\leq\abs{\int_{\mathbb
        R^3}n^{p-1}\nabla\cdot\bke{\chi(c)n\nabla c}~dx}.
    $$
Since the righthand side of the above is estimated as
    $$
        \abs{\int_{\mathbb R^3}n^{p-1}\nabla \cdot (\chi(c) n \nabla
        c)
        dx}
        \leq  C\int_{\mathbb R^3}\abs{\nabla
        n^\frac{p+\alpha}{2}}n^\frac{p-\alpha}{2}\abs{\nabla c}
        dx,
    $$
it follows from Young's inequality that
\begin{equation}\label{Lp-est-Dec10}
        \frac{1}{p}\frac{d}{dt}\norm{n}_p^p+\norm{\nabla
        n^\frac{p+\alpha}{2}}_2^2\leq C p\int_{\mathbb R^3} n^{p-\alpha}\abs{\nabla
        c}^2dx.
\end{equation}
Using H\textrm{\"{o}}lder inequality, Sobolev embedding and Young's
inequality, we estimate the integral in the righthand side of the
above as
\[
\int_{\mathbb R^3}n^{p-\alpha}\abs{\nabla c}^2
dx\leq\norm{n^{p-\alpha}}_{\frac{p}{p-\alpha}}\norm{\abs{\nabla
c}^2}_{\frac{p}{\alpha}}
\]
\[
\leq C p\norm{n}_p^{p-\alpha}\norm{\Delta
c}^2_{\frac{6p}{2p+3\alpha}}\leq
Cp\bke{\frac{\alpha}{p}+\frac{p-\alpha}{p}\norm{n}_p^p}\norm{\Delta
c}_{\frac{6p}{2p+3\alpha}}^2.
\]
Combining estimates, we get
    \begin{align*}
        \frac{d}{dt}\norm{n}_p^p&\leq C p^3\bke{\frac{\alpha}{p}+\frac{p-\alpha}{p}\norm{n}_p^p}
        \norm{\Delta c}_{\frac{6p}{2p+3\alpha}}^2\\
        &\leq C p^3\norm{\Delta
        c}^2_{\frac{6p}{2p+3\alpha}}\norm{n}_p^p+C p^2\norm{\Delta
        c}^2_{\frac{6p}{2p+3\alpha}}.
    \end{align*}
Thus it follows from Gronwall inequality that
    $$
        \norm{n}_p^p\leq\exp{\bke{c p^3\int_0^t\norm{\Delta c(s)}^2_{\frac{6p}{2p+3\alpha}}ds}}
        \int_0^t\norm{\Delta c}^2_\frac{6p}{2p+3\alpha}
        ds+\norm{n_0}_p^p
    $$
and hence \eqref{CKK100-Dec14} holds except for $p=\infty$ whenever
the following holds:
    \begin{equation}\label{eq:remains_to_show}
        \int_0^T\norm{\Delta c(s)}^2_{\frac{6p}{2p+3\alpha}}
        ds<\infty,\quad 1+\alpha<p<\infty.
    \end{equation}
We derive \eqref{eq:remains_to_show} by treating two cases, i.e.
$\alpha>1/3$ and $1/8<\alpha\leq 1/3$ separately.
\\
\\
(Case $\alpha>1/3$):\,\,
 Applying maximal regularity estimate \eqref{heat-maximal}
of heat equation to \eqref{eq:Chemotaxis-approximate}$_2$,
\[
\int_0^T\norm{\Delta
        c(s)}^2_{\frac{6p}{2p+3\alpha}}
        \leq C\bke{\norm{\nabla
        c_0}^2_{\frac{6p}{2p+3\alpha}}+\int_0^T \norm{n(s)}^2_{\frac{6p}{2p+3\alpha}}ds+\int_0^T
        \norm{u\cdot\nabla c}^2_{\frac{6p}{2p+3\alpha}}ds}
\]
\[
        :=C\bke{\norm{\nabla
        c_0}^2_{\frac{6p}{2p+3\alpha}}+\mathrm{I}+\mathrm{II}}.
\]
Now, for $p>1+\alpha,$ we obtain via interpolation
    \begin{align*}
        \mathrm{I}&\leq C\int_0^T \norm{n(s)}_1^{2-\frac{(1+2\alpha)(4p-3\alpha)}
        {2p(1+3\alpha)}}\norm{n(s)}_{3+6\alpha}^{\frac{(1+2\alpha)(4p-3\alpha)}
        {2p(1+3\alpha)}} ds\\
        &\leq C\int_0^T \norm{\nabla
        n^{\frac{1+2\alpha}{2}}(s)}_2^{\frac{4}{1+3\alpha}
        -\frac{3\alpha}{p(1+3\alpha)}}ds
        =C\int_0^T\norm{\nabla
        n^{\frac{1+2\alpha}{2}}(s)}_2^{2-\delta_p}ds
    \end{align*}
for some $\delta_p>0.$ We note that the last equality of the above
holds because$\alpha>1/3.$ Before estimating $\mathrm{II},$ we first
observe that
    \begin{equation}\label{eq:before}
        u\in L^\infty(0,T;L^6(\mathbb
        R^3)).
    \end{equation}
Indeed, since $\tau=0$, the vorticity, $\omega=\nabla\times u$,
satisfies $\omega_t-\Delta\omega=-\nabla\times(n\nabla\phi),$ and
its $L^2$ estimation becomes
    $$
        \frac{d}{dt}\norm{\omega}_2^2+\norm{\nabla\omega}_2^2\leq
        C\int_{\mathbb R^3}\nabla\omega\cdot n
        dx\leq\epsilon\norm{\nabla
        \omega}_2^2+C(\epsilon)\norm{n}_2^2.
    $$
Hence we have
\begin{equation}\label{JK-Dec11-10}
        \frac{d}{dt}\norm{\omega}_2^2\leq C\norm{n}_2^2\leq
        C\norm{n}_{1+\alpha}^{2-\frac{3+3\alpha}{2+3\alpha}}\norm{\nabla
        n^\frac{1+\alpha}{2}}_2^\frac{6}{2+3\alpha}.
\end{equation}
Since $0<\frac{6}{2+3\alpha}<2$ due to $\alpha> 1/3,$ we obtain
\eqref{eq:before} by Sobolev embedding. Next we estimate
$\mathrm{II}$. Applying maximal regularity estimate
\eqref{heat-maximal} for heat equation and \eqref{eq:before}, we
have for $p>1+\alpha,$
    \begin{align*}
        \mathrm{II}&\leq \int_0^T \norm{u(s)}_6^2\norm{\nabla
        c(s)}^2_\frac{6p}{p+3\alpha}ds\leq C\int_0^T\norm{\Delta
        c(s)}^2_\frac{2p}{p+\alpha} ds\\
        &\leq C\bke{\norm{\nabla
        c_0}^2_\frac{2p}{p+\alpha} +\int_0^T (\norm{n(s)}^2_\frac{2p}{p+\alpha}+\norm{u\cdot\nabla
        c(s)}^2_\frac{2p}{p+\alpha} ) ds}\\
        &\leq C+ C\int_0^T
        \bke{\norm{n(s)}_1^{2-\frac{3(1+2\alpha)(p-\alpha)}{2p(1+3\alpha)}}
        \norm{n(s)}_{3+6\alpha}^{\frac{3(1+2\alpha)(p-\alpha)}{2p(1+3\alpha)}}
        +\norm{\nabla c(s)}^2_\frac{6p}{2p+3\alpha}} ds\\
        &\leq C+C\int_0^T \bke{\norm{\nabla n^\frac{1+2\alpha}{2}(s)}_2^{\frac{3}{1+3\alpha}
        -\frac{3\alpha}{p(1+3\alpha)}}+\norm{\nabla
        c(s)}^2_\frac{6p}{2p+3\alpha}} ds\\
        &=:C+C\int_0^T\bke{\norm{\nabla
        n^\frac{1+2\alpha}{2}(s)}_2^{2-\delta'_p}+\norm{\nabla
        c(s)}^2_\frac{6p}{2p+3\alpha}}ds,
    \end{align*}
where
$0<\delta'_p:=\frac{-1+6\alpha}{1+3\alpha}+\frac{3\alpha}{p(1+3\alpha)}<2.$
We note here that for each $q$ with $2\leq q \leq 6,$
$\int_0^T\norm{\nabla c}_q^2 ds<\infty$ due to the results in Lemma
\ref{lemma:weak}. Therefore we conclude that $n\in
L^\infty(0,T;L^p(\mathbb R^3))$ and $\nabla n^\frac{p+\alpha}{2}\in
L^2(0,T;L^2(\mathbb R^3))$ for $\max\{1+\alpha,3\alpha\}<p<\infty,$
which can be easily extended to $1+\alpha<p<\infty$ from
\eqref{eq:from_lemma}. It remains to show that $n$ is bounded.
Indeed, since $n\in L^\infty(0,T;L^p(\mathbb R^3))$ for all $1\le
p<\infty$, we can see that $c_t$, $\nabla^2 c$, $u_t$ and $\nabla^2
u$ belong to $L^q((0,T)\times \mathbb R^3))$ for all $q<\infty$ and
therefore, we also note that $\nabla c\in L^\infty((0,T)\times
\mathbb R^3))$. Using the estimate \eqref{Lp-est-Dec10} and $\nabla
c\in L^\infty((0,T)\times \mathbb R^3))$, we obtain
\[
\frac{d}{dt}\norm{n}_p^p
\leq Cp^2\norm{n}^{p-\alpha}_{{p-\alpha}}\leq
Cp^2\norm{n}^{\frac{\alpha}{p-1}}_{1}\norm{n}^{\frac{p(p-\alpha-1)}{p-1}}_{p}
\leq Cp^2\norm{n}^{p(1-\beta)}_{p},
\]
where $\beta=\alpha/(p-1)$. Via Gronwall's inequality, we observe
that
\[
\norm{n(t)}_{p}\leq (Cp^2t)^{\frac{1}{p}}+\norm{n_0}_{p},\qquad t\le
T.
\]
Passing $p$ to the limit, we obtain for all $t<T$
\begin{equation}\label{JK-Dec10-100}
\norm{n(t)}_{L^{\infty}(\R^3)}\leq 1+\norm{n_0}_{L^{\infty}(\R^3)}.
\end{equation}
This completes the case that $\alpha>1/3$.
\\
\\
(Case $1/8<\alpha\leq 1/3$):\,\, We first show that
    \begin{equation}\label{eq:to_show}
        n\in L^\infty(0,T;L^p(\mathbb R^3))\quad\mathrm{and}\quad \nabla n^\frac{p+\alpha}{2}\in
        L^2(0,T;L^2(\mathbb R^3)),\quad 1\leq p<1+4\alpha
    \end{equation}
and we then derive \eqref{eq:remains_to_show} via
\eqref{eq:to_show}. To show \eqref{eq:to_show}, we compute $L^p$
estimate of $n$ as follows:
    $$
        \norm{n(t)}_p^p+\int_0^t\norm{\nabla
        n^\frac{p+\alpha}{2}(s)}_2^2
        ds\leq\int_0^t\int_{\mathbb R^3}
        n^{p-1}\nabla\cdot(\chi(c)n\nabla c)~dx ds+\norm{n_0}_p^p.
    $$
From Young's inequality, the first term in the righthand side of the
above is estimated as
    \begin{align*}
        &\int_0^t\int_{\mathbb R^3}
        n^{p-1}\nabla\cdot(\chi(c)n\nabla c)~ dxds=-\int_0^t\int_{\mathbb R^3}
        \nabla n^{p-1}\cdot(\chi(c)n\nabla c)~ dxds\\
        \leq &\epsilon_1\int_0^t\norm{\nabla
        n^\frac{p+\alpha}{2}}_2^2ds+C(\epsilon_1)\int_0^t\int_{\mathbb
        R^3}n^{p-\alpha}\abs{\nabla c}^2 dxds.
    \end{align*}
We note that the last integral of the above is estimated as follows:
    \begin{align*}
        &\int_0^t\int_{\mathbb
        R^3}n^{p-\alpha}\abs{\nabla c}^2
        dxds=\int_0^t\int_{\mathbb R^3} n^{p-\alpha}\nabla
        c\cdot\nabla c~dxds\\
        \leq &C\bke{\int_0^t\int_{\mathbb R^3}\abs{\nabla
        n^{p-\alpha}}\abs{\nabla c}~dxds+\int_0^t\int_{\mathbb R^3}
        \abs{n^{p-\alpha}}\abs{\Delta c}~dxds}\\
        \leq &\epsilon_2\int_0^t \norm{\nabla
        n^\frac{p+\alpha}{2}}_2^2ds +C(\epsilon_2)\int_0^t\int_{\mathbb R^3}
        n^{p-3\alpha}\abs{\nabla c}^2~dxds+C\int_0^t\int_{\mathbb
        R^3}\abs{n^{p-\alpha}}\abs{\Delta c}~dxds.
    \end{align*}
Hence choosing sufficiently small $\epsilon_1$ and $\epsilon_2>0,$
we have 
    \begin{align*}
        &\norm{n(t)}_p^p+\int_0^t\norm{\nabla
        n^\frac{p+\alpha}{2}}_2^2~ds\\
        \leq &C\bke{\int_0^t\int_{\mathbb R^3} \abs{n^{p-3\alpha}}\abs{\nabla c}^2dxds
        +\int_0^t\int_{\mathbb R^3}\abs{n^{p-\alpha}}\abs{\Delta c}
        dxds}+\norm{n_0}_p^p\\
        \leq &C\bke{\int_0^t\int_{\mathbb R^3}\abs{n^{p-3\alpha}}\abs{\nabla c}^2 dxds
        +\int_0^t(\norm{n}_{p-\alpha+1}^{p-\alpha+1}+\norm{\Delta
        c}_{p-\alpha+1}^{p-\alpha+1})
        ds}+\norm{n_0}_p^p\\
        :=&C\bke{\mathrm{I}+\mathrm{II}}+\norm{n_0}_p^p.
    \end{align*}
For each $p$ with $1+\alpha<p<1+4\alpha,$ we estimate $\mathrm{I}$
via maximal regularity estimate \eqref{heat-maximal}
    \begin{align}\label{eq:estimate_I}
        \nonumber\mathrm{I}&\leq C\int_0^t
        \norm{n^{p-3\alpha}(s)}_\frac{1+\alpha}{p-3\alpha}\norm{\abs{\nabla
        c(s)}^2}_\frac{1+\alpha}{1+4\alpha-p}ds\\
        &\leq C\int_0^t \norm{\Delta
        c(s)}_{r_1}^2ds
        \leq C\bke{\norm{\nabla
        c_0}_{r_1}^2+\int_0^t (\norm{n(s)}_{r_1}^2+\norm{u\cdot\nabla c(s)}_{r_1}^2)ds },
    \end{align}
where $r_1:=\frac{6+6\alpha}{5+14\alpha-3p}.$ We note that
H$\mathrm{\ddot{o}}$lder inequality is applicable to the first
inequality of the above due to the conditions of $\alpha$ and $p,$
that is, $1/8<\alpha\leq 1/3$ and $1<p<1+4\alpha.$ Now let us
estimate the term $\int_0^t \norm{u\cdot\nabla c}_{r_1}^2 ds$ in
\eqref{eq:estimate_I}. Considering the following $L^2$- estimation
of the vorticity equation:
    $$
        \frac{d}{dt}\norm{\omega}_2^2+\norm{\nabla\omega}_2^2\leq
        C\int_{\mathbb R^3}\nabla\omega\cdot n
        dx\leq\epsilon\norm{\nabla
        \omega}_2^2+C(\epsilon)\norm{n}_2^2,
    $$
we have
\begin{equation}\label{JK-Dec11-20}
        \frac{d}{dt}\norm{\omega}_2^2\leq C\norm{n}_2^2\leq
        C\norm{n}_{1+\alpha}^{2-\frac{3(1+2\alpha)(1-\alpha)}{2+5\alpha}}\norm{\nabla
        n^\frac{1+2\alpha}{2}}_2^\frac{6-6\alpha}{2+5\alpha}.
\end{equation}
Since $0<\frac{6-6\alpha}{2+5\alpha}<2$ due to $1/8<\alpha\leq 1/3,$
we have $$u\in L^\infty(0,T;L^6(\mathbb R^3)).$$ Taking into account
that $1\leq r_1<3$ and the above, we note, due to
H$\mathrm{\ddot{o}}$lder inequality, that
    \begin{align*}
        \int_0^t \norm{u\cdot \nabla c}_{r_1}^2 ds&\leq\int_0^t
        \norm{u}_6^2\norm{\nabla c}_{\frac{6r_1}{6-r_1}}^2 ds
        \leq C\int_0^t \norm{\Delta c}_\frac{6r_1}{6+r_1}^2 ds\\
        &\leq C\bke{\norm{\nabla
        c_0}_\frac{6r_1}{6+r_1}^2+\int_0^t (\norm{n}_\frac{6r_1}{6+r_1}^2+
        \norm{u\cdot\nabla c}_\frac{6r_1}{6+r_1}^2) ds}\\
        &\leq C\bke{\norm{\nabla
        c_0}_\frac{6r_1}{6+r_1}^2+\int_0^t (\norm{n}_\frac{6r_1}{6+r_1}^2+
        \norm{\nabla c}_{r_1}^2) ds}.
    \end{align*}
Substituting it to \eqref{eq:estimate_I} and applying interpolation
inequality, we obtain
\[
\mathrm{I}\leq C\bke{\norm{\nabla
c_0}_{\frac{6r_1}{6+r_1}}^2+\norm{\nabla
c_0}_{r_1}^2+\int_0^t(\norm{n}_{1+\alpha}^{2(1-\theta_1)}
\norm{n}_{3p+3\alpha}^{2\theta_1}
+\norm{n}_1^{2(1-\theta_2)}\norm{n}_{3p+3\alpha}^{2\theta_2}+\norm{\nabla
c}_{r_1}^2) ds},
\]
where $\theta_1=\frac{(p+\alpha)(3p-14\alpha+1)}{2(3p+2\alpha-1)}$
and
$\theta_2=\frac{3(p+\alpha)(p-3\alpha)}{2(1+\alpha)(3p+3\alpha-1)}.$
We remark that it is not difficult to verify
$1+\alpha<r_1<3p+3\alpha$ and $1<\frac{6r_1}{6+r_1}<3p+3\alpha$
whenever $1/8<\alpha\leq 1/3$ and $1+\alpha<p<1+4\alpha,$ which
enables us to apply the interpolation inequality to the above.
Therefore, we have
    \begin{align*}
        \mathrm{I}&\leq C\bke{\norm{\nabla c_0}_{\frac{6r_1}{6+r_1}}^2+\norm{\nabla
        c_0}_{r_1}^2+\int_0^t (\norm{\nabla n^\frac{p+\alpha}{2}}_2^\frac{4\theta_1}{p+\alpha}
        +\norm{\nabla n^\frac{p+\alpha}{2}}_2^\frac{4\theta_2}{p+\alpha}+\norm{\nabla
        c}_{r_1}^2)
        ds }\\
        &=: C\bke{\norm{\nabla c_0}_{\frac{6r_1}{6+r_1}}^2+\norm{\nabla
        c_0}_{r_1}^2+\int_0^t (\norm{\nabla n^\frac{p+\alpha}{2}}_2^{\delta_1}
        +\norm{\nabla n^\frac{p+\alpha}{2}}_2^{\delta_2}+\norm{\nabla
        c}_{r_1}^2)
        ds}.
    \end{align*}
Similarly, for each $p$ with $1+\alpha<p<1+4\alpha,$ we estimate
    \begin{align*}
        \mathrm{II}&\leq
        C\int_0^t(\norm{n}_{r_2}^{r_2}+\norm{u}_6^{r_2}\norm{\nabla
        c}^{r_2}_\frac{6r_2}{6-r_2})
        ds+\norm{\nabla c_0}_{r_2}^{r_2}\\
        &\leq C\bke{\int_0^t(\norm{n}_{r^2}^{r_2}+\norm{n}_\frac{6r_2}{6+r_2}^{r_2}+\norm{\nabla c}_{r_2}^{r_2}) ds+
        \norm{\nabla c_0}_{\frac{6r_2}{6+r_2}}^{r_2}+\norm{\nabla
        c_0}_{r_2}^{r_2}},
    \end{align*}
where $r_2:=p-\alpha+1.$ We note that in the first inequality of the
above, H$\mathrm{\ddot{o}}$lder inequality is applied due to the
conditions of $\alpha$ and $p,$ that is, $1/8<\alpha\leq 1/3$ and
$1+\alpha<p<1+4\alpha$. Since $1+\alpha<r_2<3p+3\alpha,$
$1<\frac{6r_2}{6+r_2}<3p+3\alpha$ and $2<r_2<6$, we obtain the
following estimate via applying interpolation inequality:
    \begin{align*}
        \mathrm{II}&\leq C\Big(\int_0^t(\norm{n}_{1+\alpha}^{r_2(1-\theta_3)}\norm{n}_{3p+3\alpha}^{r_2\theta_3}
        +\norm{n}_{1}^{r_2(1-\theta_4)}\norm{n}_{3p+3\alpha}^{r_2\theta_4})
        ds\\
        &\qquad\quad+\int_0^t(\norm{\nabla c}_{2}^{r_2(1-\theta_5)}\norm{\nabla c}_{6}^{r_2\theta_5})ds
        +\norm{\nabla c_0}_{\frac{6r_2}{6+r_2}}^{r_2}+\norm{\nabla
        c_0}_{r_2}^{r_2}\Big),
    \end{align*}
where
$$\theta_3:=\frac{3(p+\alpha)(r_2-1-\alpha)}{r_2(3p+2\alpha-1)},~\theta_4:=\frac{(p+\alpha)(5r_2-6)}{r_2(6p+6\alpha-2)}
\quad\mathrm{and}~\theta_5:=\frac{3(p-\alpha-1)}{2r_2}.$$ Therefore,
we have
    \begin{align*}
        \mathrm{II}&\leq C\bke{\int_0^t(\norm{\nabla n^\frac{p+\alpha}{2}}_{2}^{\frac{2r_2\theta_3}{p+\alpha}}
        +\norm{\nabla n^\frac{p+\alpha}{2}}_{2}^{\frac{2r_2\theta_4}{p+\alpha}}+\norm{\Delta c}_2^{r_2\theta_5})ds
        +\norm{\nabla c_0}_{\frac{6r_2}{6+r_2}}^{r_2}+\norm{\nabla
        c_0}_{r_2}^{r_2}}\\
        &=: C\bke{\int_0^t(\norm{\nabla n^\frac{p+\alpha}{2}}_{2}^{\delta_3}
        +\norm{\nabla n^\frac{p+\alpha}{2}}_{2}^{\delta_4}+\norm{\Delta c}_2^{\delta_5})ds
        +\norm{\nabla c_0}_{\frac{6r_2}{6+r_2}}^{r_2}+\norm{\nabla
        c_0}_{r_2}^{r_2}}.
    \end{align*}
Since one can see that $0<\delta_i<2,~i=1,2,...,5$ whenever
$1/8<\alpha\leq 1/3$ and $1+\alpha<p<1+4\alpha,$ it follows from
Young's inequality that
    $$
        \norm{n(t)}_p^p+C_1\int_0^t\norm{\nabla
        n^\frac{p+\alpha}{2}}_2^2 ds\leq C_2\bke{\int_0^T\norm{\nabla c}_{r_1}^2
        ds+1},\quad 0<t<T
    $$
for $p$ with $1+\alpha<p<1+4\alpha.$ We know that, for each $q$ with
$2\leq q \leq 6,$ $\int_0^T\norm{\nabla c}_q^2 ds<\infty.$ Since
$2\leq r_1\leq 6$ whenever $\frac{2+11\alpha}{3}<p<1+4\alpha,$ we
have $n\in L^\infty(0,T;L^p(\mathbb R^3))$ and $\nabla
n^\frac{p+\alpha}{2}\in L^2(0,T;L^2(\mathbb R^3)),$ for
$\frac{2+11\alpha}{3}<p<1+4\alpha,$ which can be easily extended to
$1+\alpha<p<1+4\alpha$ from \eqref{eq:from_lemma}. Thus we conclude
that \eqref{eq:to_show} holds.

We are now ready to derive \eqref{eq:remains_to_show} for the case
$1/8<\alpha\leq 1/3.$ Let $p_0=\frac{3}{2}-\frac{3\alpha}{4}.$ Since
$1\leq p_0<1+4\alpha$ whenever $1/8<\alpha\leq 1/3,$ it is evident
from the above conclusion that
    \begin{equation}\label{eq:evident}
        n\in L^\infty(0,T;L^{p_0}(\mathbb
        R^3)).
    \end{equation}
From the maximal regularity estimate \eqref{heat-maximal} for heat
equation, we have
    \begin{align*}
        \int_0^T\norm{\Delta c}_{\frac{6p}{2p+3\alpha}}^{2}
        &\leq C\bke{\int_0^T \norm{n}_{\frac{6p}{2p+3\alpha}}^{2}ds+
        \int_0^T\norm{u\cdot\nabla c}_{\frac{6p}{2p+3\alpha}}^{2}
        ds}\\
        &=: C\bke{\mathrm{I}+\mathrm{II}}.
    \end{align*}
For $p>1+\alpha,$ we have $p_0<\frac{6p}{2p+3\alpha}<3p_0+3\alpha$
and hence it follows from \eqref{eq:evident} that
    \begin{align*}
        \mathrm{I}&\leq C\int_0^T \norm{n(s)}_{p_0}^{2-\frac{(p_0+\alpha)(6p-2pp_0-3\alpha p_0)}
        {2(2p_0+3\alpha)}}\norm{n(s)}_{3p_0+3\alpha}^{\frac{(p_0+\alpha)(6p-2pp_0-3\alpha p_0)}
        {2(2p_0+3\alpha)}} ds\\
        &\leq C\int_0^T \norm{\nabla
        n^{\frac{p_0+\alpha}{2}}(s)}_2^{\frac{12-4p_0}{2p_0+3\alpha}
        -\frac{6\alpha p_0}{p(2p_0+3\alpha)}}ds\\
        &=C\int_0^T\norm{\nabla
        n^{\frac{p_0+\alpha}{2}}(s)}_2^{2-\delta_p}ds
    \end{align*}
for some $\delta_p>0.$ On the other hands, by the maximal regularity
estimate \eqref{heat-maximal} for the heat equation, we have for
$p>\frac{\alpha(1+\alpha)}{1-\alpha}$ (in fact, $p>1+\alpha$ since
$0<\frac{\alpha(1+\alpha)}{1-\alpha}<1+\alpha$)
    \begin{align*}
        \mathrm{II}&\leq \int_0^T \norm{u(s)}_6^2\norm{\nabla
        c(s)}^2_\frac{6p}{p+3\alpha}ds\leq C\int_0^T\norm{\Delta
        c(s)}^2_\frac{2p}{p+\alpha} ds\\
        &\leq C\bke{\norm{\nabla c_0}^2_\frac{2p}{p+\alpha}+\int_0^T (\norm{n(s)}^2_\frac{2p}{p+\alpha}+\norm{u\cdot\nabla
        c(s)}^2_\frac{2p}{p+\alpha} )ds}\\
        &\leq C\bke{\norm{\nabla c_0}^2_\frac{2p}{p+\alpha}+\int_0^T
        (\norm{n(s)}_{1+\alpha}^{2-\frac{3(1+2\alpha)(p-\alpha p-\alpha-\alpha^2)}{p(2+5\alpha)}}
        \norm{n(s)}_{3+6\alpha}^{\frac{3(1+2\alpha)(p-\alpha p-\alpha-\alpha^2)}{p(2+5\alpha)}}
        +\norm{\nabla c(s)}^2_\frac{6p}{2p+3\alpha}) ds}\\
        &\leq C\bke{\norm{\nabla c_0}^2_\frac{2p}{p+\alpha}+
        \int_0^T (\norm{\nabla n^\frac{1+2\alpha}{2}(s)}_2^{\frac{6-6\alpha}{2+5\alpha}
        -\frac{6\alpha+6\alpha^2}{p(2+5\alpha)}}+\norm{\nabla
        c(s)}^2_\frac{6p}{2p+3\alpha}) ds}\\
        &=:C\bke{\norm{\nabla c_0}^2_\frac{2p}{p+\alpha}+\int_0^T(\norm{\nabla
        n^\frac{1+2\alpha}{2}(s)}_2^{2-\delta'_p}+\norm{\nabla
        c(s)}^2_\frac{6p}{2p+3\alpha})ds},
    \end{align*}
where
$0<\delta'_p:=\frac{-2+16\alpha}{2+5\alpha}+\frac{6\alpha+6\alpha^2}{p(2+5\alpha)}<2.$
We remark that $\int_0^T\norm{\nabla c}_q^2 ds<\infty$ for each $q$
with $2\leq q \leq 6$, due to the results in Lemma \ref{lemma:weak}.

So far, we have shown \eqref{eq:remains_to_show} holds for each
$\alpha>1/8.$ Therefore we conclude that $n\in
L^\infty(0,T;L^p(\mathbb R^3))$ for $1+\alpha<p<\infty.$ Following
similar procedures as in \eqref{JK-Dec10-100}, we can show that
$L^{\infty}$-norm of $n$ is bounded. Since arguments are on the same
track, we omit the details.
\end{proof}

\section{Proofs of Theorems}
To prove Theorem \ref{thm:3D-existence} and Theorem
\ref{thm:3D-L^infinite bound}, using the uniform estimates
established previously, we construct weak and bounded weak
solutions. Using the uniform estimates established in the previous
section. Since the argument is rather standard (compare to
\cite{CKL,FLM,TW_1,TW_2,V}), we omit the details and give the sketch
of how our constructions are made instead.
\\
\begin{main-pfs}
We consider only the case of Theorem \ref{thm:3D-L^infinite bound},
since the proof of Theorem \ref{thm:3D-existence} is essentially
same. First, we recall the regularized system \eqref{eq:Chemotaxis}
with the initial data $(n_{0\varrho},c_{0,\varrho},u_{0\varrho})$
which are chosen as smooth approximations of $(n_0,c_0,u_0)$:
    \begin{align*}
        n_{0\varrho}=\psi_\varrho\ast n_0,\quad c_{0\varrho}=\psi_\varrho\ast c_0
        \quad\mathrm{and}\quad u_{0\varrho}=\psi_\varrho\ast u_0,
    \end{align*}
where $\phi_\varrho$ denotes the usual mollifier. The convergence of
$(n_{0\varrho},c_{0\varrho},u_{0\varrho})$ entails that the
estimates obtained in Lemma \ref{lemma:weak} and Lemma
\ref{lemma:bdd-weak} are uniform, independent of $\varrho$,
precisely, the constant $C$ and $M$ in \eqref{eq:E'+D<CE} can be
chosen independent of $\varrho.$ Likewise, there exists a constant
$C$ such that for $q<\infty$
\begin{equation}\label{CKK1000-Dec16}
\norm{n_{\varrho}}_{L^{\infty}((0,T)\times \R^3)}+\norm{\nabla
n^{\frac{q+\alpha}{2}}_{\varrho}}_{L^2((0,T)\times
\R^3)}<C,
\end{equation}
\begin{equation}\label{CKK2000-Dec16}
\norm{c_{\varrho}}_{L^\infty(0,T;W^{1,q}(\mathbb
R^3))}+\norm{c_{\varrho}}_{L^q(0,T;W^{2,q}(\mathbb
R^3))}+\norm{\partial_t c_{\varrho}}_{L^q(0,T;L^q(\mathbb
R^3))}<C,
\end{equation}
\begin{equation}\label{CKK3000-Dec16}
\norm{u_{\varrho}}_{L^\infty(0,T;W^{1,q}(\mathbb
R^3))}+\norm{u_{\varrho}}_{L^q(0,T;W^{2,q}(\mathbb
R^3))}+\norm{\partial_t u_{\varrho}}_{L^q(0,T;L^q(\mathbb
R^3))}<C.
\end{equation}
According to the estimates we have derived, a bootstrap argument can
extend the local solution to any given time interval $(0,T)$
(compare to \cite{FLM}, \cite{TW_2} and \cite{SK} for more detail).
Let $k$ be any number with $k\geq 2+\alpha$. We then show that
$\partial_t n_{\varrho}$ and $\partial_t n^k_{\varrho}$ are,
independent of $\varrho$, in $L^1(0,T;W^{-2,2}(\mathbb R^3))$, where
$W^{-2,2}(\mathbb R^3)$ is the dual space of $W^{2,2}(\mathbb R^3)$
(compare to \cite{TW_2}). Then via Aubin-Lions Lemma, by passing to
the limit, we have some weak limit $(n,c,u)$, which turns out to be
a weak solution. Its verification is rather straightforward, and
thus the details are skipped. It is also direct that $(n,c,u)$ is a
bounded weak solution and satisfies the estimates
\eqref{CKK1000-Dec16}-\eqref{CKK3000-Dec16}. This completes the
proof.
\end{main-pfs}

Next we present the proofs of Theorem \ref{bounded-domain-10} and
Theorem \ref{bounded-domain-20}. As mentioned in the Introduction,
we indicate only difference compared to the case of $\R^3$.

\begin{main-pfs-domains}
We note that unlike $\R^3$, $L^1$ estimate of $n\langle x\rangle$ is
not necessary, because negative part of $\int_{\Omega}n\log n$ is
controlled by $\norm{n}_{L^1(\Omega)}$. We also observe that
Gagliardo-Nierenberg inequality should be slightly modified, for
example, we used in the case $\R^3$ (see the inequality right above
\eqref{eq:A4})
\[
\norm{n}_{L^2(\R^3)}^2\leq
C\norm{n}_{L^1(\R^3)}^\frac{1+6\alpha}{2+6\alpha}\norm{\nabla
n^\frac{1+2\alpha}{2}}_{L^2(\R^3)}^\frac{6}{2+6\alpha}.
\]
In the case of bounded domains, it is replaced by
\[
\norm{n}_{L^2(\Omega)}^2\leq
C\norm{n}_{L^1(\Omega)}^\frac{1+6\alpha}{2+6\alpha}\norm{\nabla
n^\frac{1+2\alpha}{2}}_{L^2(\Omega)}^\frac{6}{2+6\alpha}+C\norm{n}_{L^1(\Omega)}^2.
\]
Major modifications lie in the estimate of vorticity, $\omega$,
because the boundary condition of $\omega$ is not prescribed. More
precisely, in order to obtain \eqref{eq:before}, we used the
equation of vorticity, which is not useful to the case of bounded
domains. Here we show \eqref{eq:before} differently not by using
vorticity equations. Let $Q=\Omega\times (0,T)$. Using $L^p-$type
estimate of the Stokes system (see e.g. \cite{GS}), we note that
\begin{equation}\label{JK-Dec-100}
\norm{u_t}_{L^2(Q)}+\norm{u}_{L^2((0,T);H^{2}(\Omega))}+\norm{\nabla
p}_{L^2(Q)}\leq C\bke{\norm{n}_{L^2(Q)}+\norm{u_0}_{H^{1}(\Omega)}}.
\end{equation}
Testing $-\Delta u$ to the fluid equations,
\[
\frac{1}{2}\frac{d}{dt}\int_{\Omega}\abs{\nabla
u}^2dx+\int_{\Omega}\abs{\Delta u}^2dx\leq
\frac{1}{4}\int_{\Omega}\abs{\Delta u}^2dx+C\int_{\Omega}
n^2dx+\int_{\Omega}\abs{\Delta u}\abs{\nabla p}dx
\]
\[
\leq \frac{1}{2}\int_{\Omega}\abs{\Delta u}^2dx+C\int_{\Omega}
n^2dx+C\int_{\Omega}\abs{\nabla p}^2dx.
\]
Therefore, after integrating the above in time over $(0, t)$ for any
$t<T$ and combining the estimate \eqref{JK-Dec-100}, we obtain
\[
\norm{\nabla u(t)}^2_{L^2(\Omega)}+\norm{\Delta u}^2_{L^2(Q_t)}\leq
\norm{\nabla u_0}^2_{L^2(\Omega)}
+C\bke{\norm{n}^2_{L^2(Q_t)}+\norm{\nabla p}^2_{L^2(Q_t)}}
\]
\[
\leq C\bke{\norm{u_0}_{W^{1,2}(\Omega)}+\norm{n}^2_{L^2(Q_t)}},
\]
where $Q_t=(0,t)\times \Omega$ and we used \eqref{JK-Dec-100}. We
note that $L^2$ norm of $n$ can be estimated in the same ways as in
\eqref{JK-Dec11-10} and \eqref{JK-Dec11-10}, which implies that
$\nabla u\in L^{\infty}((0,T);L^2(\Omega))$ and therefore, it is
automatic that $u\in L^{\infty}((0,T);L^6(\Omega))$ via Sobolev
embedding. The rest parts of proofs are essentially the same as the
cases of whole space, and thus we omit the details.
\end{main-pfs-domains}

\section*{Acknowledgments}
Yun-Sung Chung's work is supported by NRF-2012R1A1A2001373 and
Kyungkeun Kang's work is supported by NRF-2014R1A2A1A11051161.

\begin{equation*}
\left.
\begin{array}{cc}
{\mbox{Yun-Sung Chung}}\qquad&\qquad {\mbox{Kyungkeun Kang}}\\
{\mbox{Department of Mathematics }}\qquad&\qquad
 {\mbox{Department of Mathematics}} \\
{\mbox{Yonsei University
}}\qquad&\qquad{\mbox{Yonsei University}}\\
{\mbox{Seoul, Republic of Korea}}\qquad&\qquad{\mbox{Seoul, Republicof Korea}}\\
{\mbox{ysjung93@hanmail.net }}\qquad&\qquad
{\mbox{kkang@yonsei.ac.kr }}
\end{array}\right.
\end{equation*}

\end{document}